\documentclass[11pt]{article}

\usepackage{pdfsync}
\usepackage{amsthm}
\usepackage{amsmath}
\usepackage{amssymb}
\usepackage{xcolor}
\usepackage{hyperref}

\usepackage{verbatim}
\usepackage[margin=1in]{geometry}

\usepackage[shortlabels]{enumitem}
\setlist{nolistsep, noitemsep}

\newtheorem{theorem}{Theorem}[section]
\newtheorem{lemma}[theorem]{Lemma}
\newtheorem{problem}[theorem]{Problem}
\newtheorem{proposition}[theorem]{Proposition}
\newtheorem{corollary}[theorem]{Corollary}
\newtheorem{conjecture}[theorem]{Conjecture}

\theoremstyle{definition}
\newtheorem{definition}[theorem]{Definition}

\newcommand{\meas}[1]{\ensuremath{\mu\left(#1\right)}}

\newcommand{\edgemult}[1]{\ensuremath{\left|#1\right|}}

\newcommand{\Prob}[1]{\ensuremath{%
    \mathbb P\left[#1\right]
  }}
\DeclareMathOperator{\mad}{mad}

\title{Fractional coloring with local demands and applications to degree-sequence bounds on the independence number}

\author{Tom Kelly%
  \thanks{
    School of Mathematics, Georgia Institute of Technology. 
    Email: \protect\href{mailto:tom.kelly@gatech.edu}{\protect\nolinkurl{tom.kelly@gatech.edu}}.
    Research
supported by the National Science Foundation under Grant No. DMS-224707
  }
  \and
  Luke Postle%
  \thanks{
    Department of Combinatorics and Optimization, University of Waterloo, Canada.
    Email: \protect\href{mailto:lpostle@uwaterloo.ca}{\protect\nolinkurl{lpostle@uwaterloo.ca}}.
    Partially supported by NSERC under Discovery Grant No.\ 2019-04304, the Ontario Early Researcher Awards program and the Canada Research Chairs program.
  }}

\begin{document}
\maketitle

\begin{abstract}
  In a fractional coloring, vertices of a graph are assigned measurable subsets of the real line and adjacent vertices receive disjoint subsets; the fractional chromatic number of a graph is at most $k$ if it has a fractional coloring in which each vertex receives a subset of $[0, 1]$ of measure at least $1/k$.  We introduce and develop the theory of ``fractional colorings with local demands'' wherein each vertex ``demands'' a certain amount of color that is determined by local parameters such as its degree or the clique number of its neighborhood.  This framework provides the natural setting in which to generalize degree-sequence type bounds on the independence number.  Indeed, by Linear Programming Duality, all of the problems we study have an equivalent formulation as a problem concerning weighted independence numbers, and they often imply new bounds on the independence number.  

  Our results and conjectures are inspired by many of the most classical results and important open problems concerning the independence number and the chromatic number, often simultaneously.  We conjecture a local strengthening of both Shearer's bound on the independence number of triangle-free graphs and the fractional relaxation of Molloy's recent bound on their chromatic number, as well as a longstanding problem of Ajtai et al.\ on the independence number of $K_r$-free graphs and the fractional relaxations of Reed's $\omega, \Delta, \chi$ Conjecture and the Total Coloring Conjecture.   We prove an approximate version of the first two, and we prove ``local demands'' versions of Vizing's Theorem and of some $\chi$-boundedness results.
\end{abstract}

\section{Introduction}\label{intro section}

Let $G$ be a graph.  A \textit{proper coloring} of $G$ is an assignment of colors to the vertices of $G$ such that adjacent vertices receive different colors, and the \textit{chromatic number} of $G$, denoted $\chi(G)$, is the fewest number of colors needed to properly color $G$.  An \textit{independent set} in $G$ is a subset of vertices of $G$ that are pairwise nonadjacent, and the \textit{independence number} of $G$, denoted $\alpha(G)$, is the size of a largest independent set in $G$.  These graph parameters are among the oldest and most well-studied in graph theory.  The \textit{fractional chromatic number} of $G$, denoted $\chi_f(G)$ and defined formally in Definition~\ref{fractional coloring definition}, lies between them in the following sense:

\begin{equation}\label{coloring parameters inequality}
  |V(G)|/\alpha(G) \leq \chi_f(G) \leq \chi(G).
\end{equation}

There are two popular streams of research concerning the fractional chromatic number.  On the one hand, we can sometimes prove better upper bounds for $\chi_f$ when the same bound is impossible or out of reach for $\chi$.  For example, Reed's Conjecture~\cite{R98} is known to hold for the fractional chromatic number (see~\cite{MR02}), Kilakos and Reed~\cite{KR93} proved the fractional relaxation of the Total Coloring Conjecture~\cite{B65, V68}, and Reed and Seymour~\cite{RS98} proved that the fractional chromatic number of graphs with no $K_{t + 1}$-minor is at most $2t$, a factor of two away from the bound famously conjectured by Hadwiger~\cite{H43}.  In fact, in the first paper on fractional coloring,  before the Four Color Theorem was proved~\cite{AH76}, Hilton, Rado, and Scott~\cite{HRS73} proved that planar graphs have fractional chromatic number strictly less than five.  See also~\cite{DH18}.

On the other hand, the problem of generalizing bounds on the independence number to the fractional chromatic number has received considerable attention.  Dvo\v{r}\'{a}k, Sereni, and Volec~\cite{DSV14} proved that triangle-free graphs of maximum degree three have fractional chromatic number at most $14/5$, generalizing a well-known result of Staton~\cite{S79} and resolving a conjecture of Heckman and Thomas~\cite{HT01}.  Heckman and Thomas~\cite{HT06} proved that triangle-free subcubic planar graphs on $n$ vertices have independence number at least $3n/8$, and they also conjectured that this result can be generalized to the fractional chromatic number.  See also~\cite{DM17},~\cite{DSV15}, and~\cite[Conjecture 4.3]{CvBdJdVKP18}.

In this paper we expand both of these streams in a novel way by developing the theory of ``fractional coloring with local demands''.  This framework provides both the natural setting in which to generalize degree-sequence type bounds on the independence number as well as new methods to prove new independence number bounds.  Moreover, ``local demands'' results are more robust versions of bounds on the fractional chromatic number, and they imply results about ``fractional coloring with local list sizes,'' which can be useful for precoloring extension problems in the fractional setting.
We begin with a more thorough introduction to fractional coloring and the local demands framework.  Our main results and conjectures are stated in Section~\ref{results section}.

\subsection{Fractional coloring}\label{fractional coloring intro}

Before introducing the ``local demands'' theory, we need some definitions.  The fractional chromatic number can be defined from several different perspectives, each with its own advantages.  We prefer the following definition because of its similarity to proper coloring.  Here, and throughout the paper, we use $\mu$ to denote the Lebesgue measure on the real numbers.
\begin{definition}\label{fractional coloring definition}
  Let $G$ be a graph.
  \begin{itemize}
  \item A \textit{fractional coloring} of $G$ is a function $\phi$ with domain $V(G)$ such that for each $v\in V(G)$, the image $\phi(v)$ is a measurable subset of $\mathbb R$ such that for each $uv\in E(G)$, we have $\phi(u)\cap \phi(v) = \varnothing$.
  \item The \textit{fractional chromatic number} of $G$, denoted $\chi_f(G)$, is the infimum over all positive real numbers $k$ such that $G$ has a fractional coloring such that for every $v\in V(G)$, we have $\phi(v) \subseteq [0, k]$ and $\meas{\phi(v)} \geq 1$.
  \end{itemize}
\end{definition}

Each way to define the fractional chromatic number has a natural, more general ``local analogue'' which we can use in the ``local demands'' setting, and Dvo\v{r}\'{a}k, Sereni, and Volec \cite[Theorem 2.1]{DSV14} proved that these local analogues are also equivalent, as we now discuss.  We primarily use the following notation of Dvo\v{r}\'{a}k, Sereni, and Volec~\cite{DSV14}, the local analogue of Definition~\ref{fractional coloring definition}.
\begin{definition}\label{f-coloring definition}
  Let $G$ be a graph.
  \begin{itemize}
  \item A \textit{demand function} for $G$ is a function $f : V(G) \rightarrow [0, 1]\cap\mathbb Q$.
  \item If $f$ is a demand function for $G$, an \textit{$f$-coloring} of $G$ is a fractional coloring $\phi$ such that for every $v\in V(G)$, we have $\phi(v)\subseteq [0, 1]$ and $\meas{\phi(v)} \geq f(v)$.
  \end{itemize}
\end{definition}

Note that $\chi_f(G) \leq k$ if and only if $G$ admits an $f$-coloring when $f(v) = 1/k$ for each $v\in V(G)$.  In the local demands theory, rather than upper bounding the fractional chromatic number, we seek to prove the existence of $f$-colorings for certain natural demand functions $f$ that are not constant.  Before discussing these demand functions, let us introduce some other notions that are equivalent to fractional coloring.

Dvo\v{r}\'{a}k, Sereni, and Volec \cite{DSV14} also defined a discrete analogue of an $f$-coloring, as follows.

\begin{definition}\label{common denominator definition}
  Let $G$ be a graph with demand function $f$.
  \begin{itemize}
  \item An integer $N$ is a \textit{common denominator} for $f$ if $N\cdot f(v)$ is an integer for every $v\in V(G)$.
  \item If $N$ is a common denominator for $f$, then an \textit{($f, N)$-coloring} of $G$ is an assignment $\psi$ of subsets of $\{1, \dots, N\}$ to the vertices of $G$ such that for every $uv\in E(G)$, we have $\psi(u)\cap \psi(v) = \varnothing$ and for every $v\in V(G)$, we have $|\psi(v)| \geq N\cdot f(v)$.
  \end{itemize}
\end{definition}

The concept of an $(f, N)$-coloring is the local analogue of what is commonly referred to as an \textit{$(a : b)$-coloring}, which is an $(f, N)$-coloring in which $a = N$ and $f(v)\cdot N = b$ for each $v\in V(G)$.  Note that if $G$ has an $(f, N)$-coloring for some common denominator $N$, then $G$ has an $f$-coloring.

Fractional coloring is also interesting from the perspective of polyhedral combinatorics.  The \textit{stable set polytope} of a graph $G$ is the convex hull of the incidence vectors of the independent sets of $G$ in $\mathbb R^{|V(G)|}$.  The fractional chromatic number is often defined as the optimum solution to a certain Linear Program, and the solution to this LP is at most $k$ if and only if the vector in $\mathbb R^{|V(G)|}$ in which each entry is $1/k$ is in the stable set polytope of $G$.  A graph $G$ has an $(f, N)$-coloring for some common denominator $N$ if and only if the vector of demands $(f(v) : v \in V(G))$ is in the stable set polytope of $G$.  Hence, the local analogue of this LP-based definition of $\chi_f$ involves determining if the vector of demands is in the stable set polytope.  Note also that the vector of demands is in the stable set polytope of $G$ if and only if there exists a probability distribution on the independent sets of $G$ such that if $I$ is selected according to this distribution, then $\Prob{v\in I} \geq f(v)$ for each $v\in V(G)$.

The lower bound on $\chi_f$ in~\eqref{coloring parameters inequality} actually holds more generally, as follows.  If a graph $G$ has an $f$-coloring, then for any nonnegative \textit{weight function} $w : V(G)\rightarrow \mathbb R_+$, there is an independent set $I\subseteq V(G)$ such that $\sum_{v\in I}w(v) \geq \sum_{v\in V(G)}w(v)f(v)$.  The converse also holds, as Dvo\v{r}\'{a}k, Sereni, and Volec~\cite[Theorem~2.1]{DSV14} demonstrated using LP-duality that $G$ has this latter property if and only if the vector of demands is in the stable set polytope.  Hence we have the following equivalent ways to view fractional coloring.

\begin{proposition}\label{equivalent definitions}
  Let $G$ be a graph with demand function $f$.  The following are equivalent.
  \begin{enumerate}[(a)]
  \item The graph $G$ has an $f$-coloring.
  \item There exists a common denominator $N$ for $f$ such that $G$ has an $(f, N)$-coloring.
  \item\label{equiv-definitions-prob-dist} There exists a probability distribution on the independent sets of $G$ such that if $I$ is selected according to the distribution, then $\Prob{v\in I} \geq f(v)$ for each $v\in V(G)$.
  \item The vector of demands $(f(v) : v \in V(G))$ is in the stable set polytope of $G$.
  \item\label{weighted independence number} For every nonnegative weight function $w : V(G) \rightarrow \mathbb R_+$, the graph $G$ contains an independent set $I$ such that $\sum_{v\in I}w(v) \geq \sum_{v\in V(G)}w(v)f(v)$.
  \end{enumerate}
\end{proposition}

Thus by~\ref{weighted independence number}, our fractional coloring results in this paper imply bounds on the independence number, almost all of which are new, and in some cases, the more general formulation as a fractional coloring problem is essential to the proof.  In fact, if a graph $G$ has an $f$-coloring, then for any $k$, there is an induced $k$-colorable subgraph in $G$ with at least $\sum_{v\in V(G)} 1 - (1 - f(v))^k$ vertices, and consequently every graph $G$ contains an induced $k$-colorable subgraph with at least $(1 - ((\chi_f(G) - 1)/\chi_f(G))^k)|V(G)|$ vertices.  (To prove this fact, consider choosing $x_1, \dots, x_k \in [0, 1]$ independently and uniformly at random; if $\phi$ is an $f$-coloring of $G$, then the set $\{v \in V(G) : \phi(v) \cap \{x_1, \dots, x_k\} \neq \varnothing\}$ induces a $k$-colorable subgraph in $G$ and has size at least $\sum_{v\in V(G)} 1 - (1 - f(v))^k$ in expectation).

Researchers have previously considered fractional coloring with respect to non-constant demand functions in order to obtain upper bounds on $\chi_f$ in one of at least two ways, as follows.
If $f$ and $g$ are demand functions for a graph $G$ such that $G$ has an $f$-coloring and a $g$-coloring, and $\lambda \in [0, 1]$, then $G$ has a $(\lambda \cdot f + (1 - \lambda)\cdot g)$-coloring.  In particular, if $\lambda f(v) + (1 - \lambda)g(v) \geq 1/k$ for each $v\in V(G)$, then $\chi_f(G) \leq k$.  This technique is common in fractional coloring; it is used in~\cite{DSV15, EK13, KKS10, KKK11}, and it is used implicitly in Kilakos and Reed's~\cite{KR93} proof of the fractional relaxation of the Total Coloring Conjecture.  In Section~\ref{tri-free section}, we use this technique to show that our Conjecture~\ref{local tri-free}, if true, implies a recent conjecture of Cames van Batenburg et al.~\cite[Conjecture 4.3]{CvBdJdVKP18} on the fractional chromatic number of triangle-free graphs.
Another way that considering non-constant demand functions can be useful for bounding $\chi_f$ is that these demand functions can be used to prove a statement stronger than simply a bound on $\chi_f$, which in turn can make using induction easier.  Dvo\v{r}\'{a}k, Sereni, and Volec~\cite{DSV14} first used this idea in their proof that triangle-free subcubic graphs have fractional chromatic number at most $14/5$, and they introduced Definitions~\ref{f-coloring definition} and~\ref{common denominator definition} for this purpose.  These results contrast with local demands in that the demand functions being considered are technical by nature and only interesting for their application to bounding $\chi_f$.

We also introduce list coloring in the fractional setting, as follows.
\begin{definition}\label{fractional list definition}
  Let $G$ be a graph.
  \begin{itemize}
  \item If $L$ is a function with domain $V(G)$ such that $L(v)$ is
a measurable subset of $\mathbb R$ for each $v\in V(G)$,
    then $L$ is a \textit{fractional list-assignment} for $G$.
  \item A \textit{fractional $L$-coloring} is a fractional coloring $\phi$ of $G$ such that every vertex $v\in V(G)$ satisfies $\phi(v)\subseteq L(v)$ and $\meas{\phi(v)} \geq 1$.
  \end{itemize}
\end{definition}

There is a natural way to define a fractional list chromatic number using Definition~\ref{fractional list definition}; however, this parameter would simply equal the fractional chromatic number, since the following proposition implies that if $\chi_f(G) \leq k$, then $G$ has a fractional $L$-coloring for any fractional list-assignment $L$ satisfying $\meas{L(v)} \geq k$ for each $v\in V(G)$.
\begin{proposition}\label{local fractional list implication}
  Let $G$ be a graph with demand function $f$ such that $G$ has an $f$-coloring.  If $L$ is a fractional list-assignment for $G$ such that each $v\in V(G)$ satisfies $\meas{L(v)} \geq f(v)^{-1}$, then $G$ has a fractional $L$-coloring.
\end{proposition}

We prove a more general version (Lemma~\ref{list coloring lemma}) of Proposition~\ref{local fractional list implication} in Section~\ref{fractional list section}.

Although the fractional list chromatic number is no different from the fractional chromatic number, it is interesting to study ``fractional coloring with local list sizes'', in which we seek to prove the existence of fractional $L$-colorings when the measure of each vertex's list is determined by its local structure.  Davies et al.~\cite[Theorem~2]{DdJdVKP18} recently considered a special case of such a problem for triangle-free graphs.  Such results are useful for extending fractional precolorings, as demonstrated by Corollary~\ref{weak list TCC}.  By Proposition~\ref{local fractional list implication}, these problems can potentially be proved in a stronger form in the setting of local demands.  
  
\subsection{Local Demands}
The aim of this paper is to introduce and explore fractional colorings of graphs with respect to demand functions in which the demand for each vertex depends naturally on local parameters, such as its degree.
Besides being independently interesting, there are two main reasons to consider these ``local demands'' problems.

\begin{enumerate}[(1)]
\item Using local demands is the natural way in which to generalize degree-sequence type bounds on the independence number to fractional coloring, via Proposition~\ref{equivalent definitions} (e).  In particular, if we know that for some class of graphs $\mathcal G$, every $G\in\mathcal G$ satisfies $\alpha(G) \geq \sum_{v\in V(G)}f_G(v)$ for some demand function $f_G$, then we seek to determine if every $G\in\mathcal G$ has an $f_G$-coloring.
\item Local demands results provide more robust versions of bounds on the fractional chromatic number, since $\chi_f(G) \leq \max_{v\in V(G)} f(v)^{-1}$ if $G$ has an $f$-coloring -- if $\mathcal G$ is a class of graphs such that every $G\in\mathcal G$ satisfies $\chi_f(G) \leq g(G)$ for some function $g$, then we seek to determine if $G$ has an $f$-coloring where $f(v) = g(G[N[v]])$, where $N[v]$ is the closed neighborhood of $v$.  Moreover, as demonstrated by Proposition~\ref{local fractional list implication}, local demands results have implications in the setting of fractional coloring with local list sizes.
\end{enumerate}

As these reasons demonstrate, this new paradigm parallels classical research on bounds for both the chromatic number and the independence number of graphs.  Many well-known problems and results have natural generalizations in the local demands setting; we conjecture the following:
\begin{itemize}
\item Conjecture~\ref{local tri-free}, the local fractional strengthening of both Shearer's~\cite{Sh83, Sh91} bounds on the independence number of triangle-free graphs and the fractional relaxation of Molloy's~\cite{M17} bound on their chromatic number, which is closely related to the asymptotics of the off-diagonal Ramsey number $R(3, k)$~\cite{BK13, FPGM13},
\item Conjecture~\ref{local K_r-free}, the local fractional strengthening of a longstanding problem of Ajtai, Erd\H os, Koml\'{o}s, and Szemer\'{e}di~\cite{AEKS81} on the independence number of $K_r$-free graphs,
\item Conjecture~\ref{local reeds conj}, the local fractional version of Reed's $\omega, \Delta, \chi$ Conjecture~\cite{R98}, and
\item Conjecture~\ref{local fractional TCC}, the local fractional version of the Total Coloring Conjecture~\cite{B65, V68}.
\end{itemize}
\vspace{\baselineskip}
In this paper, we make the following contributions toward developing the theory of local demands.
\begin{itemize}
\item We generalize the result of Alon, Tuza, and Voigt~\cite{ATV97} that the fractional list chromatic number equals the fractional chromatic number to the setting of local demands in Section~\ref{ATV section} (Theorem~\ref{local multiplicative coloring}).
\item We provide three different proofs of the ``local fractional greedy bound'' (Theorem~\ref{caro wei theorem}), which simultaneously implies both the Caro-Wei Theorem~\cite{C79, W81} and the fractional relaxation of the ``greedy bound'' on the chromatic number in Section~\ref{caro wei proof section}.
\item We make progress towards both our local fractional Shearer/Molloy Conjecture (Theorem~\ref{approximate tri-free}) and our local fractional Ajtai et al.\ Conjecture (Theorem~\ref{approximate Kr}) as well as a local fractional version of a recent result of Bonamy et al.~\cite{BKNP18} (Theorem~\ref{local low omega}) on the chromatic number of graphs without large cliques in Section~\ref{tri-free section}.
\item We confirm the local fractional Reed's Conjecture for perfect graphs (Theorem~\ref{local frac perfect}), which strengthens a recent result of Brause et al.~\cite{CRRS16}.  We actually prove a stronger local fractional result (Theorem~\ref{chi-bounded meta-theorem}) about $\chi$-bounded classes of graphs with a linear $\chi$-binding function in Section~\ref{chi-bounded section}.
\item We prove the local fractional version of Vizing's Theorem~\cite{V64} (Theorem~\ref{fractional generalized vizings}) in Section~\ref{edge coloring section}.
\end{itemize}
We give precise formulations of these problems and further discussion in Section~\ref{results section}.

\section{Results and conjectures}\label{results section}

\subsection{Greedy coloring}

As mentioned, fractional coloring with local demands is the natural setting to which to generalize degree-sequence type bounds on the independence number.  There are numerous such bounds on the independence number, for example~\cite{CRRS16, CT91, G83, HM19, HR11, HS01, HS06, S94, Sh91, Th99}, and the concept of local demands enables us to generalize many of these to the setting of fractional coloring.  The archetypal example comes from the famous Caro-Wei Theorem~\cite{C79, W81}, which states that every graph $G$ satisfies $\alpha(G) \geq \sum_{v\in V(G)}{1}/({d(v) + 1})$, where $d(v)$ is the degree of $v$.   This result actually strengthens Tur\'{a}n's Theorem~\cite{T41}, which is equivalent to the following.  If $G$ is a graph on $n$ vertices with average degree $d$, then $\alpha(G) \geq n/(d + 1)$.  Tur\'{a}n's Theorem and the Caro-Wei Theorem can be generalized to fractional coloring using local demands, as follows.

\begin{theorem}[Local Fractional Greedy Bound]\label{caro wei theorem}
  If $G$ is a graph with demand function $f$ such that $f(v) \leq 1/ (d(v) + 1)$ for each $v\in V(G)$, then $G$ has an $f$-coloring.
\end{theorem}

We remark that the dual formulation of Theorem~\ref{caro wei theorem} (using Proposition~\ref{equivalent definitions}~\ref{weighted independence number}) was proved in~\cite{STY03}.

As we see in Section~\ref{caro wei proof section}, there are several different ways to prove this theorem.  All of these proofs are reminiscent of proofs of the so-called ``greedy bound'' on the chromatic number, which states that every graph $G$ satisfies $\chi(G) \leq \Delta(G) + 1$, where $\Delta(G)$ is the maximum degree of $G$.  This is no coincidence, as Theorem~\ref{caro wei theorem} also implies the fractional relaxation of greedy bound, that is that every graph $G$ satisfies $\chi_f(G) \leq \Delta(G) + 1$.  

\subsection{Triangle-free graphs}

In 1980, Ajtai, Koml\'{o}s, and Szemer\'{e}di~\cite{AKS80} famously proved that every triangle-free graph $G$ on $n$ vertices with average degree $d$ has independence number at least $0.01(n/d)\log d$.  Shearer~\cite{Sh83} later improved the leading constant to $(1 - o(1))$, and in~\cite{Sh91}, he generalized this by proving there is an independent set of size at least $\sum_{v\in V(G)}(1 - o(1))\log d(v)/d(v)$.  (We use $\log$ to denote the natural logarithm.) In the 90s, Johansson~\cite{J96-tri} proved the related result that every triangle-free graph $G$ satisfies $\chi(G) = O(\Delta(G)/\log \Delta(G))$, and recently, Molloy~\cite{M17} improved the leading constant to $(1 + o(1))$.  We conjecture the local demands version of both Shearer's result and the fractional relaxation of Molloy's result, as follows.

\begin{conjecture}[Local Fractional Shearer/Molloy]\label{local tri-free}
  If $G$ is a triangle-free graph with demand function $f$ such that $f(v) \leq (1 - o(1))\log d(v) / d(v)$ for each $v\in V(G)$, then $G$ has an $f$-coloring.
\end{conjecture}

It is possible that the bounds of Shearer~\cite{Sh83, Sh91} and Molloy~\cite{M17} for triangle-free graphs can be improved.  A result of Bollob\' as~\cite{B81} implies that there exist triangle-free graphs on $n$ vertices and maximum degree $d$ for large $n$ and $d$ such that $\alpha(G) \leq 2n\log d / d$, and no asymptotically better upper bound is known.  Note that these graphs have chromatic number at least $d / (2\log d)$, which is the best known lower bound for large $d$.  Determining the best possible upper bounds for $|V(G)|/\alpha(G)$, $\chi_f(G)$, and $\chi(G)$ in terms of the average or maximum degree of $G$ when $G$ is a triangle-free graph is an important open problem.  
If true, Conjecture~\ref{local tri-free} simultaneously implies both Shearer's~\cite{Sh91} bound on the independence number and the fractional relaxation of Molloy's~\cite{M17} bound on the chromatic number of triangle-free graphs, which are the current best known upper bounds for $|V(G)|/\alpha(G)$ in terms of the average degree and for $\chi_f(G)$ in terms of the maximum degree, respectively.  Moreover, proving Conjecture~\ref{local tri-free} necessitates a new approach, which may provide insight into the barrier to improving the bounds for $\alpha, \chi_f$, and $\chi$.  Improving any of these bounds would be a major breakthrough.

The \textit{Ramsey number} $R(\ell, k)$ is the smallest $n$ such that every graph on at least $n$ vertices contains a clique of size $\ell$ or an independent set of size $k$.  Using either result of Shearer~\cite{Sh83, Sh91} or Molloy's~\cite{M17} result, it is straightforward to show that $R(3, k) \leq (1 + o(1))k^2/\log k$.  An improvement to the leading constant in any of these results would also improve this bound on the Ramsey number $R(3, k)$.  In 1995, Kim~\cite{K95ramsey} proved that $R(3, k) = \Omega(k^2/\log k)$, and in 2013, Fiz Pontiveros, Griffiths, and Morris~\cite{FPGM13} and independently Bohman and Keevash~\cite{BK13} proved a lower bound of $(1/4 - o(1))k^2/\log k$ for $R(3, k)$.  Determining the leading constant of the Ramsey number $R(3, k)$ asymptotically is also a major open problem.  The best known upper bound on $R(3, k)$ is equivalent to the following fact: if $G$ is a triangle-free graph on $n$ vertices, then $n / \alpha(G) \leq (\sqrt 2 + o(1))\sqrt{n / \log n}$.  Cames van Batenburg et al.~\cite[Conjecture 4.3]{CvBdJdVKP18} recently conjectured that this bound on $n/\alpha$ can be generalized to fractional coloring, and we show in Section~\ref{tri-free section} that Conjecture~\ref{local tri-free} is a natural strengthening of their conjecture.

In Section~\ref{tri-free section}, we also make progress towards Conjecture~\ref{local tri-free} by proving the following result.

\begin{theorem}\label{approximate tri-free}
  If $G$ is a triangle-free graph with demand function $f$ such that
  \begin{equation*}
    f(v) \leq (2e - o(1))^{-1}\frac{d(v)\log\log d(v)}{\log d(v)}
  \end{equation*}
  for each $v\in V(G)$, then $G$ has an $f$-coloring.
\end{theorem}

By Proposition~\ref{local fractional list implication}, Theorem~\ref{approximate tri-free} implies that if $G$ is a graph of sufficently large minimum degree with fractional list-assignment $L$ satisfying $\meas{L(v)} \geq \log d(v) / (6 d(v)\log\log d(v))$ for each $v\in V(G)$, then $G$ has a fractional $L$-coloring.  Interestingly, this result does not hold in the ordinary list coloring setting as long as the maximum degree is not bounded by a function of the minimum degree, as shown by Davies et al.~\cite[Proposition~11]{DdJdVKP18}.  Davies et al.~\cite[Theorem~2]{DdJdVKP18} also proved that if $G$ is a triangle-free graph with fractional list-assignment $L$ such that $L(v) = [0, (1 + o(1))d(v)/\log d(v)]$, then $G$ has a fractional $L$-coloring.  It would be interesting to prove the stronger result that every triangle-free graph $G$ with fractional list-assignment $L$ satisfying $\meas{L(v)} \geq (1 + o(1))d(v)/\log d(v)$ for each $v\in V(G)$ has a fractional $L$-coloring, which would follow from Conjecture~\ref{local tri-free}.

\subsection{$K_r$-free graphs}

Ajtai, Erd\H os, Koml\'{o}s, and Szemer\'{e}di~\cite{AEKS81} proved that for some constant $c$, if $G$ is a $K_r$-free graph of average degree $d$, then $|V(G)|/\alpha(G) \leq cd / \log ((\log d) / r)$.  Shearer~\cite{Sh95} later proved that for any fixed $r$, if $G$ is a $K_r$-free graph such that $\Delta(G) \leq \Delta$, then $|V(G)| / \alpha(G) = O\left(\Delta\log\log\Delta/\log\Delta\right)$.  Since any graph $G$ has an induced subgraph on at least $|V(G)|/2$ vertices of maximum degree at most twice the average degree of $G$, Shearer's result actually holds with $\Delta(G)$ replaced by the average degree and thus improves the bound of Ajtai et al.~\cite{AEKS81} significantly.  Johansson~\cite{J96-Kr} generalized Shearer's bound for the chromatic number, and Molloy~\cite{M17} recently found a shorter proof of Johansson's result.  Ajtai et al.\ suggested that the best upper bound for any fixed $r$ may actually be of the form $cd / \log d$, that is the $\log\log\Delta$ term in Shearer's result is not needed.  This bound is currently widely believed to hold.  In the same vein as the Caro-Wei Theorem and Shearer's~\cite{Sh91} result, It is natural to ask if a degree-sequence version of this bound holds and if it holds more generally for fractional coloring.  Thus, we conjecture the following.

\begin{conjecture}[Local Fractional Ajtai-Erd\H os-Koml\'{o}s-Szemer\'{e}di]\label{local K_r-free}
  For every $r$, there exists some constant $c$ such that the following holds.  If $G$ is a $K_r$-free graph with demand function $f$ such that $f(v) \leq c\log d(v) / d(v)$ for each $v\in V(G)$, then $G$ has an $f$-coloring.
\end{conjecture}

In Section~\ref{tri-free section}, we make progress towards this conjecture by proving the following.
\begin{theorem}\label{approximate Kr}
  For every $r$, there exists some constant $c$ such that the following holds.  If $G$ is a $K_r$-free graph with demand function $f$ such that
  \begin{equation*}
    f(v) \leq \frac{c\log d(v)}{d(v)(\log\log d(v))^2}
  \end{equation*}
  for each $v\in V(G)$, then $G$ has an $f$-coloring.
\end{theorem}

Theorem~\ref{approximate Kr} implies that for every $r$, there exists $c$ such that if $G$ is a $K_r$-free graph, then $\alpha(G) \geq \sum_{v\in V(G)}c\log d(v) / (d(v)(\log\log d(v))^2)$.  No degree-sequence bound was previously known, and this bound improves Shearer's~\cite{Sh95} for graphs that are far from regular.  This bound is also better than the bound of Ajtai et al.~\cite{AEKS81}.  We do not know how to obtain this degree-sequence bound on the independence number of $K_r$-free graphs without considering fractional coloring.  

\subsection{Graphs without large cliques}

We now consider $K_r$-free graphs where $r$ is no longer fixed.  We let $\omega(G)$ denote the size of a largest clique in $G$, and for each vertex $v\in V(G)$, we let $\omega(v) = \omega(G[N(v)])$.
Recently, Bonamy, Kelly, Nelson, and Postle~\cite{BKNP18} proved that if $\Delta$ is sufficiently large and $G$ is a graph with $\Delta(G) \leq \Delta$ and $\omega(G) \leq \Delta^{1/(72c)^2}$, then $\chi(G) \leq \Delta / c$.  We prove the following local demands analogue of this result in Section~\ref{tri-free section}.

\begin{theorem}\label{local low omega}
  If $G$ is a graph with demand function $f$ and $c \geq e^6$ such that each vertex $v\in V(G)$ satisfies $\omega(v) \leq d(v)^{1/(432c\log c)^2}$ and $f(v) \leq c/d(v)$, then $G$ has an $f$-coloring.
\end{theorem}

We remark that it was not previously known and is nontrivial to show that Theorem~\ref{caro wei theorem} can be strengthened for triangle-free graphs of large minimum degree -- that is, if $G$ is a triangle-free graph with demand function $f$ such that each $v\in V(G)$ satisfies $f(v) \leq (1 + \varepsilon - o(1))/(d(v) + 1)$ for some $\varepsilon > 0$, then $G$ has an $f$-coloring.  This fact follows from Theorem~\ref{local low omega} with a considerably more relaxed assumption concerning the cliques in $G$.  It would be interesting to improve Theorem~\ref{local low omega} by showing that every graph has an $f$-coloring where $f(v) \leq c/d(v)$ if $\omega(v) \leq d(v)^{1/\Omega(c)}$ for each vertex $v$.  This result would be best possible by the following Ramsey-theoretic result of Spencer~\cite{S77}: there exist graphs on $n$ vertices for large $n$ with independence number $c$ (and thus fractional chromatic number at least $n/c$) such that every clique has size at most $n^{\frac{2}{c+2} + o(1)}$.

We actually derive Theorems~\ref{approximate tri-free},~\ref{approximate Kr}, and~\ref{local low omega} by first proving the following result, which functions as a ``black box'' that converts a bound on the fractional chromatic number into a result about fractional coloring with local demands (and consequently a bound on the independence number).  We will apply this theorem with $(1 + o(1))^{-1}\log d$, $\log d / (200\omega\log\log d)$, and $6c\log c$ playing the role of $r(d)$ to obtain Theorems~\ref{approximate tri-free},~\ref{approximate Kr}, and~\ref{local low omega}, respectively.  Essentially, we show that for a graph $G$, if every subgraph $H$ of $G$ induced by vertices of degree at most $\Delta$ in $G$ has fractional chromatic number at most $\Delta / r(\Delta)$, then $G$ has an $f$-coloring if $f$ is a demand function satisfying $f(v) \leq (2e + o(1))r(d(v)) / (d(v)\log r(d(v)\cdot r(d(v))))$.  Note that here the demand for a vertex $v$ in terms of its degree is similar to the hypothesis about the fractional chromatic number of the subgraphs of $G$, with an additional factor of $(2e + o(1)) / \log r(d(v)\cdot r(d(v)))$.

\begin{theorem}\label{approximate local demands version}
  For every $\varepsilon > 0$, there exists $\delta > 0$ such that the following holds.  Let $r : \mathbb N \rightarrow \mathbb R_{> 1}$ be nondecreasing such that 
  \begin{equation}\label{r constraint}
    \lim_{n\rightarrow\infty}\frac{r(n\cdot r(n))}{e\cdot r(n)\cdot\log(r(n\cdot r(n)))} = 0,
  \end{equation}
  and let $G$ be a graph such that $\chi_f(H) \leq \Delta / r(\Delta)$ for every sufficiently large $\Delta$ and every subgraph $H\subseteq G$ such that $\max_{v\in V(H)}d_G(v) \leq \Delta$.  If $G$ has demand function $f$ such that
  \begin{equation*}
    f(v) \leq \min\left\{(2e + \varepsilon)^{-1}\frac{r(d(v))}{d(v)\log r(d(v)\cdot r(d(v)))}, \delta\right\}
  \end{equation*}
  for each $v\in V(G)$, then $G$ has an $f$-coloring.
\end{theorem}

We prove Theorem~\ref{approximate local demands version} using a new technique that we call ``stratification''.  In order to find an $f$-coloring of $G$, we partition the vertices of $G$ into ``strata'' and color each strata in turn.  The proof suggests that a more refined approach involving list coloring and ``color degrees'' could be used to resolve Conjectures~\ref{local tri-free} and~\ref{local K_r-free}; in Section~\ref{fractional color degree section}, we state Conjecture~\ref{local color degree conjecture}, which generalizes Molloy's bound, and prove that Conjecture~\ref{local color degree conjecture} implies Conjecture~\ref{local tri-free}.

\subsection{Reed's Conjecture}

Fajtlowicz~\cite{F78, F84} proved that every graph $G$ satisfies $|V(G)|/\alpha(G) \leq (\Delta(G) + \omega(G) + 1)/2$, where $\omega(G)$ is the size of the largest clique in $G$.  
Reed's Conjecture~\cite{R98} states that this bound can be generalized to the chromatic number up to rounding, that is every graph $G$ satisfies $\chi(G) \leq \lceil (\Delta(G) + 1 + \omega(G))/2\rceil$.  This conjecture is currently a major open problem in graph coloring; however, it is a well-known folklore result that its fractional relaxation holds (for example, see~\cite{MR02}).  In fact, the rounding is not necessary, that is $\chi_f(G) \leq (\Delta(G) + 1 + \omega(G))/2$ for every graph $G$, so this bound generalizes Fajtlowicz's~\cite{F78, F84} bound on the independence number.  We conjecture the ``local demands'' version of this result, as follows.

\begin{conjecture}[Local Fractional Reed's]\label{local reeds conj}
  If $G$ is a graph with demand function $f$ such that
  \begin{equation*}
    f(v) \leq \frac{2}{d(v) + \omega(v) + 1}
  \end{equation*}
  for each $v\in V(G)$, then $G$ has an $f$-coloring.
\end{conjecture}

Brause et al.~\cite{CRRS16} conjectured that every graph $G$ satisfies $\alpha(G) \geq \sum_{v\in V(G)}2/(d(v) + \omega(v) + 1)$, and Conjecture~\ref{local reeds conj} generalizes their conjecture.  Note that Theorem~\ref{local low omega} implies that for some $\varepsilon > 0$, Conjecture~\ref{local reeds conj} and the conjecture of Brause et al.~\cite{CRRS16} hold if each vertex $v$ satisfies $\omega(v) \leq d(v)^{\varepsilon}$.  One of the main results in~\cite{CRRS16} is that perfect graphs satisfy this conjectured bound on the independence number.  In Section~\ref{chi-bounded section}, we strengthen this result by proving that Conjecture~\ref{local reeds conj} holds in a stronger sense for perfect graphs.  A graph is \textit{perfect} if every induced subgraph has chromatic number equal to its clique number.  We prove the following.

\begin{theorem}\label{local frac perfect}
  If $G$ is a perfect graph with demand function $f$ such that $f(v) \leq 1/\omega(v)$ for each $v\in V(G)$, then $G$ has an $f$-coloring.
\end{theorem}

A class of graphs is \textit{$\chi$-bounded} with \textit{$\chi$-binding function} $g$ if every induced subgraph $H$ of a graph in the class satisfies $\chi(H) \leq g(\omega(H))$.  Thus, perfect graphs form the family of graphs with the identity as the $\chi$-binding function.  We actually derive Theorem~\ref{local frac perfect} by proving a more general result (Theorem~\ref{chi-bounded meta-theorem}) about $\chi$-bounded classes of graphs with a linear $\chi$-binding function, and we derive several other interesting results (for example, Corollaries~\ref{quasiline chi-bounded} and \ref{claw-free chi-bounded}) of a similar flavor.

\subsection{Edge and total coloring}

Line graphs also form a $\chi$-bounded class of graphs, which leads us to Section~\ref{edge coloring section}, in which we prove results about edge coloring.  Vizing's Theorem~\cite{V64}, one of the most classical results in graph coloring, states that every graph $G$ can be edge-colored with at most $\Delta(G) + 1$ colors, and Vizing actually generalized this result to coloring line graphs of multigraphs.  We prove the ``local demands'' version of the generalized Vizing's Theorem~\cite{V64}, as follows.

\begin{theorem}[Local Fractional Generalized Vizing's]\label{fractional generalized vizings}
    If $G$ is a multigraph and $f$ is a demand function for the line graph $L(G)$ such that each $e\in V(L(G))$ with $e = uv\in E(G)$ satisfies $f(e) \leq 1 / (\max\{d(u), d(v)\} + \edgemult{uv})$, where $\edgemult{uv}$ is the number of edges of $G$ incident to both $u$ and $v$, then $L(G)$ has an $f$-coloring.
\end{theorem}

For any graph $G$, the \textit{total graph} of $G$, denoted $T(G)$, is the graph with vertices $V(G)\cup E(G)$ where a vertex $v\in V(G)$ is adjacent in $T(G)$ to every $u\in N(v)$ and every edge incident to $v$, and an edge $uv\in E(G)$ is adjacent in $T(G)$ to $u, v$, and every edge incident to $u$ or $v$.  The famous Total Coloring Conjecture~\cite{B65} states that every graph $G$ satisfies $\chi(T(G)) \leq \Delta(G) + 2$.  Kilakos and Reed~\cite{KR93} proved the fractional relaxation of the Total Coloring Conjecture.  We conjecture the local demands version of this conjecture, as follows.
\begin{conjecture}[Local Fractional Total Coloring Conjecture]\label{local fractional TCC}
  If $G$ is a graph and $f$ is a demand function for $T(G)$ such that each $v\in V(G)$ satisfies $f(v) \leq 1/(d(v) + 2)$ and each $uv \in E(G)$ satisfies $f(uv) \leq 1/(\max\{d(u), d(v)\} + 2)$, then $T(G)$ has an $f$-coloring.
\end{conjecture}

By Theorem~\ref{caro wei theorem}, if $G$ is a graph and $f$ is a demand function for $T(G)$ such that $f(v) \leq 1/(2d(v) + 1)$ for each vertex $v\in V(G)$ and $f(uv) \leq 1/(d(u) + d(v) + 1)$ for each edge $uv\in E(G)$, then $T(G)$ has an $f$-coloring.  It would be interesting to prove the weakening of Conjecture~\ref{local fractional TCC} in which each vertex $v$ demands $f(v) = 1/(d(v) + C)$ and each edge $uv$ demands $f(uv) = 1/(\max\{d(u), d(v)\} + C)$ for some constant $C$.  The consequence of this result in the setting of ``fractional coloring with local list sizes'' follows easily from Theorem~\ref{fractional generalized vizings} with $C = 3$, and it implies that for every graph $G$, the total graph $T(G)$ satisfies $\chi_f(T(G)) \leq \Delta(G) + 3$.  This result is a classic ``easy example'' of a bound on $\chi_f$ (see for example~\cite[Proposition~4.6.3]{SU11}), and the machinery developed in this paper yields the following streamlined proof.

\begin{corollary}[Local-list Fractional Weak TCC]\label{weak list TCC}
  If $G$ is a graph and $L$ is a fractional list-assignment for $T(G)$ such that $\meas{L(v)} \geq d(v) + 1$ for each vertex $v\in V(G)$ and $\meas{L(uv)} \geq \max\{d(u), d(v)\} + 3$ for each edge $uv\in E(G)$, then $T(G)$ has a fractional $L$-coloring.
\end{corollary}
\begin{proof}
  By Theorem~\ref{caro wei theorem} and Proposition~\ref{local fractional list implication}, $G$ has a fractional $L$-coloring $\phi$.  For each edge $uv \in E(G)$, let $L'(uv) = L(uv)\setminus (\phi(u)\cup \phi(v))$, and note that $\meas{L'(uv)} \geq \max\{d(u), d(v)\} + 1$.  Thus, by Theorem~\ref{fractional generalized vizings} and Proposition~\ref{local fractional list implication}, $L(G)$ has an $f$-coloring $\phi'$.  We can combine $\phi$ and $\phi'$ to obtain a fractional $L$-coloring of $T(G)$, as desired.
\end{proof}

It would also be interesting to consider total colorings of graphs of high girth, as in~\cite{KKK11, KKS10}.

\subsection{Brooks' Theorem}

In a companion paper~\cite{KP19-brooks} (see also \cite{KP19-ind}), we consider the analogue of Brooks' Theorem~\cite{B41} in the setting of fractional coloring with local demands, and we significantly improve Theorem~\ref{caro wei theorem}.  We prove that if $G$ is a graph with demand function $f$ such that $f(v) \leq 1/(d(v) + 1/2)$ for each vertex $v\in V(G)$ and every clique $K\subseteq V(G)$ satisfies $\sum_{v\in K}f(v) \leq 1$, then $G$ has an $f$-coloring.  As a corollary, we obtain new bounds on the independence number for graphs in which less than half of the vertices in each clique are simplicial.  In~\cite{KP19-brooks}, we also propose local demands analogues of results about graphs with chromatic number close to the maximum degree~\cite{FMR05, MR14, R99}, notably the Borodin-Kostochka Conjecture~\cite{BK77}.

\section{Fractional coloring with lists}\label{list section}

In this section, we discuss two different ways in which one may consider a notion of list coloring in the fractional coloring setting.  As illustrated by Definitions~\ref{f-coloring definition} and~\ref{common denominator definition}, there is a dichotomy between the ``discrete'' and the ``analytic'' perspectives in fractional coloring, and each of the two ways to consider list coloring corresponds to one of these perspectives.

\subsection{Discrete list versions}\label{ATV section}

First we discuss list coloring from the discrete perspective.  We need the classical concept of \textit{list coloring}, which was introduced as a generalization of proper coloring independently by Erd\H os, Rubin, and Taylor~\cite{ERT80} and Vizing~\cite{V76} in the 1970s.

\begin{definition}
  Let $G$ be a graph, and let $L = (L(v) \subseteq \mathbb N : v\in V(G))$ be a collection of ``lists of colors''.  
  \begin{itemize}
  \item If $L(v)$ is non-empty for each vertex $v\in V(G)$, then $L$ is a \textit{list-assignment} for $G$, and if $|L(v)| \geq k$ for every $v\in V(G)$, then $L$ is a \textit{$k$-list-assignment}.
  \item An \textit{$L$-coloring} of $G$ is a proper coloring $\phi$ of $G$ such that $\phi(v)\in L(v)$ for every vertex $v\in V(G)$, and $G$ is \textit{$L$-colorable} if there is an $L$-coloring of $G$.
  \end{itemize}
  The \textit{list chromatic number} of $G$, denoted \textit{$\chi_\ell(G)$}, is the smallest $k$ such that $G$ is $L$-colorable for any $k$-list-assignment $L$.  If $\chi_\ell(G) \leq k$, then $G$ is \textit{$k$-list-colorable}.
\end{definition}

Erd\H os, Rubin, and Taylor~\cite{ERT80} called the list chromatic number the \textit{choice number} and a $k$-list colorable graph \textit{$k$-choosable}.  They introduced the following definition that is related to fractional coloring: a graph $G$ is \textit{$(a : b)$-choosable} if for every $a$-list-assignment $L$, there is a map $\phi$ that assigns to each vertex $v\in V(G)$ a $b$-subset of $L(v)$ such that $\phi(u)\cap \phi(v) = \varnothing$ for each $uv\in E(G)$.  If $G$ is $(a : b)$-choosable, then $G$ has an $(f, b)$-coloring where $f$ is a demand function for $G$ such that $f(v) = a/b$ for each $v\in V(G)$.  Erd\H os, Rubin, and Taylor~\cite{ERT80} posed the following question regarding $(a : b)$-choosability.

\begin{problem}[Erd\H os, Rubin, and Taylor \cite{ERT80}]\label{multiplicative list-coloring problem}
  If a graph $G$ is $(a, b)$-list-colorable, is it $(am, bm)$-list-colorable for every positive integer $m$?
\end{problem}

Recently, Dvo\v{r}\'{a}k, Hu, and Sereni~\cite{DHS18} resolved Problem~\ref{multiplicative list-coloring problem} with an answer of ``no'': they proved the existence of a 4-choosable graph that is not $(8 : 2)$-choosable.  However, Alon, Tuza, and Voigt~\cite{ATV97} proved that Problem~\ref{multiplicative list-coloring problem} is true in a strong sense for large $m$, as follows.

\begin{theorem}[Alon, Tuza, and Voigt \cite{ATV97}]\label{multiplicative coloring implies choosability}
  If a graph $G$ is $(a, b)$-colorable, then there exists an integer $m > 1$ such that $G$ is $(am, bm)$-list-colorable.
\end{theorem}

In Section~\ref{fractional coloring intro}, we discussed a notion of the fractional list chromatic number from the analytic perspective.  We can also define a fractional list chromatic number from the discrete perspective, as follows.

\begin{definition}
  Let $G$ be a graph with demand function $f$, and let $N$ be a common denominator for $f$.
  \begin{itemize}
  \item If $L$ is an $N$-list-assignment for $G$, an \textit{$f$-fold $L$-coloring} of $G$ is an assignment $\psi$ of subsets of $L(v)$ to the vertices of $G$ such that for every $uv\in E(G)$, $\psi(u)\cap \psi(v) = \varnothing$ and for every $v\in V(G)$, $|\psi(v)| \geq N\cdot f(v)$.
  \item The graph $G$ is \textit{$(f, N)$-list-colorable} if it has an $f$-fold $L$-coloring for every $N$-list-assignment $L$.
  \end{itemize}
  The \textit{fractional list chromatic number} is the infimum over all positive real numbers $k$ such that $G$ is $(f, N)$-list-colorable where $f$ is a demand function for $G$ such that $f(v) = 1/k$ for each $v\in V(G)$ and $N$ is a common denominator for $f$.
\end{definition}

Theorems~\ref{equivalent definitions} and~\ref{multiplicative coloring implies choosability} imply that this version of the fractional list chromatic number is also equal to the fractional chromatic number.  It is natural to wonder if this situation changes if we consider demand functions that are not constant.  In this subsection, we show that Theorem~\ref{multiplicative coloring implies choosability} holds for any demand function, as follows.

\begin{theorem}\label{local multiplicative coloring}
  Let $G$ be a graph with demand function $f$, and let $N$ be a common denominator for $f$.  If $G$ has an $(f, N)$-coloring, then there exists a common denominator $M > N$ for $f$ such that $G$ is $(f, M)$-list-colorable.
\end{theorem}

The proof of Theorem~\ref{local multiplicative coloring} is similar to the proof of Theorem~\ref{multiplicative coloring implies choosability}.  We need the following two lemmas from~\cite{ATV97}.

\begin{lemma}[Alon, Tuza, and Voigt \cite{ATV97}]\label{ATV lemma 3.1}
  Let $(n_i : i \in I)$ be a sequence of positive integers, where each $n_i$ is at most $k$.  Let $M$ and $N$ be two integers and suppose that $\sum_{i\in I}n_i = M$, that $M/N$ is divisible by all integers up to $k$, and that $k\cdot \mathrm{lcm}(2, 3, \dots, k) \leq M/N$, where $\mathrm{lcm}(2, 3, \dots, k)$ denotes the least common multiple of $2, 3, \dots, k$.  Then there is a partition $I = I_1 \cup I_2 \cup \cdots \cup I_N$ of $I$ into $N$ pairwise disjoint sets such that for every $j\in\{1, \dots, N\}$, $\sum_{i\in I_j}n_i = M/N$.
\end{lemma}

For the next lemma, we need the following definitions.  A hypergraph $H = (X, F)$ is \textit{$\ell$-uniform} if each of its edges contain precisely $\ell$ vertices.  If $R$ is a subset of the vertex set of $H$, then let $H_R$ denote the hypergraph with vertex set $R$ and edge set $\{e \cap R : e \in F\}$

\begin{lemma}[Alon, Tuza, and Voigt \cite{ATV97}]\label{ATV lemma 3.2}
  If $H = (X, F)$ is a uniform hypergraph with $n$ edges, then there is a partition $X = \cup_{i\in I}X_i$ of $X$ into pairwise disjoint sets such that $H_{X_i}$ is $n_i$-uniform and $n_i \leq (n+1)^{(n+1)/2}$ for every $i\in I$.
\end{lemma}

\begin{proof}[Proof of Theorem~\ref{local multiplicative coloring}]
  Let $n = |V(G)|$, let $k = (n + 1)^{(n + 1)/2}$, and choose $M$ to be divisible by all integers up to $\max\{k^4, N^4\}$.

  We show that $G$ is $(f, M)$-list-colorable.  To that end, let $L$ be an $M$-list-assignment for $G$.  It suffices to show that $G$ has an $f$-fold $L$-coloring.  Let $H$ be the hypergraph with vertex set $X = \cup_{v\in V(G)}L(v)$ and edges $F = (L(v) : v\in V(G))$.  By Lemma~\ref{ATV lemma 3.2}, there is a partition $(X_i : i\in I)$ of the colors $X$ so that for each $i$ and each $v\in V(G)$,
  \begin{equation*}
    |L(v)\cap X_i| = n_i,
  \end{equation*}
  where $n_i \leq k$.  Note that $\sum_{i\in I}n_i = M$.  By Lemma~\ref{ATV lemma 3.1}, there is a partition $I_1, \dots, I_N$ of $I$ such that for each $j\in\{1, \dots, N\}$, $\sum_{i\in I_j}n_i = M / N$.

  By assumption, $G$ has an $(f, N)$-coloring $\psi$.  For each $v\in V(G)$, let
  \begin{equation*}
    \psi'(v) = \bigcup_{j\in \psi(v)}\cup_{i\in I_j}\{L(v)\cap X_i\}.
  \end{equation*}
  Now,
  \begin{equation*}
    |\psi'(v)| = \sum_{j\in \psi(v)} \sum_{i\in I_j}n_i = \sum_{j\in\psi(v)} M/N = |\psi(v)|M/N = f(v)\cdot M,
  \end{equation*}
  and $\psi'(u)$ and $\psi'(v)$ are disjoint when $u$ and $v$ are adjacent.  Hence, $\psi'$ is an $f$-fold $L$-coloring, as desired.
\end{proof}

\subsection{Fractional list coloring}\label{fractional list section}

In this subsection, we discuss list coloring from the analytic perspective, as in Definition~\ref{fractional list definition}.  The concepts we introduce in this subsection are useful in the setting of fractional coloring with local demands; some of the results and conventions introduced in this subsection are needed in Sections~\ref{caro wei proof section} and \ref{tri-free section} and in~\cite{KP19-brooks}.  We need the following local analogue of Definition~\ref{fractional list definition}.

\begin{definition}
  If $G$ is a graph with fractional list-assignment $L$, a \textit{$(g, L)$-coloring} of $G$ is a fractional coloring of $G$ such that every vertex $v\in V(G)$ satisfies $\phi(v)\subseteq L(v)$ and $\meas{\phi(v)} \geq g(v)$.
\end{definition}

In Sections~\ref{caro wei proof section} and \ref{tri-free section} and in~\cite{KP19-brooks}, we often find a fractional coloring of an induced subgraph of a graph $G$ and try to extend it to all of $G$.  To that end, we introduce the following notation.

\begin{definition}
  Let $G$ be a graph with fractional list-assignment $L$, and let $S \subseteq V(G)$.  If $\phi$ is a fractional $L$-coloring of $G[S]$, then we let \textit{$L_\phi$} be the fractional list-assignment for $G - S$ where for each $v\in V(G - S)$,
    \begin{equation*}
      L_\phi(v) = L(v) \setminus \bigcup_{u\in S\cap N(v)}\phi(u).
    \end{equation*}
\end{definition}

If the fractional list-assignment $L$ in the above definition is not specified, then we assume $L(v) = [0, 1]$ for each vertex $v$.

The following proposition is self-evident.

\begin{proposition}\label{extending partial coloring}
  Let $G$ be a graph with demand function $f$ and fractional list-assignment $L$.  If for some $S\subseteq V(G)$, $\phi$ is an $(f|_S, L)$-coloring of $G[S]$ such that $G - S$ has a fractional $(f|_{V(G)\setminus S}, L_\phi)$-coloring, then $G$ has an $(f, L)$-coloring.\hfill$\square$
\end{proposition}

If $\phi$ is a fractional coloring of $G[S]$ for some $S\subseteq V(G)$ and $u\in V(G)$, then $v$ \textit{sees the color} $\bigcup_{u\in S\cap N(v)}\phi(u)$, and if $\mu(\bigcup_{u\in S\cap N(v)}\phi(u)) \leq \alpha$, then $v$ \textit{sees at most $\alpha$ color}.  If $\phi'$ is a fractional coloring of $G[S']$ for some $S'\subseteq V(G)$ such that $S\subseteq S'$ and $\phi = \phi|_S$, then $\phi'$ \textit{extends} $\phi$.

In fractional coloring, one can partition the real line or the $[0, 1]$-interval as finely as needed and find fractional colorings in each part separately.  The following lemma makes use of this idea.  

\begin{lemma}\label{list coloring lemma}
  Let $G$ be a graph with demand function $f$, fractional list-assignment $L$, and $g : V(G)\rightarrow \mathbb R$.  If $g(v) \leq f(v)\mu(L(v))$ for each $v\in V(G)$ and for each $S\subseteq V(G)$ such that $\meas{\cap_{v\in S}L(v)} > 0$, the graph $G[S]$ has an $f$-coloring, then $G$ has a fractional $(g, L)$-coloring.
\end{lemma}
\begin{proof}
  By possibly resizing and shifting $L$ and scaling $g$, we may assume without loss of generality that $L(v)\subseteq [0, 1]$ for each vertex $v$.
  
  For each $S\subseteq V(G)$, let
  \begin{equation*}
    C_S = (\cap_{v\in S}L(v))\setminus(\cup_{v\in V(G)\setminus S}L(v)),
  \end{equation*}
  and for each $v\in S$, let $L_S(v) = C_S$ and $f_S = f(v)\cdot \meas{C_S}$.  Note that if $S \neq S'$, then $C_S \cap C_{S'} = \varnothing$.  By assumption, $G[S]$ has an $f$-coloring, so $G[S]$ has an $(f_S, L_S)$-coloring $\phi_S$.  For each $v\in V(G)$, let $\phi(v) = \cup_{S\ni v}\phi_S(v)$.  Now $\phi$ is a fractional $(g, L)$-coloring, as desired.
\end{proof}

Lemma~\ref{list coloring lemma} implies that if a graph $G$ has an $(f/c)$-coloring for some $c\in (0, 1)$, then $G$ has an $(f, L)$-coloring for any fractional list-assignment $L$ in which $\meas{L(v)} \geq c$ for each $v\in V(G)$.  That is, the ``worst'' uniform fractional list-assignment is the one which assigns the same list to every vertex.  In the proof of the main result of~\cite{KP19-brooks}, we encounter a uniform fractional list-assignment, and we go to considerable length in the proof to ensure that vertices have different lists.  

Lemma~\ref{list coloring lemma} also implies Proposition~\ref{local fractional list implication}.
It is interesting to study fractional colorings with respect to fractional list-assignments that are not necessarily uniform.  These types of problems naturally have applications to ``precoloring extension''-type problems in the fractional setting.  Indeed, we prove some such results below (Lemmas~\ref{halls} and \ref{coloring K_n minus a matching}), and we use these results in~\cite{KP19-brooks}.

In~\cite{KP19-brooks}, we often use the following lemma of Edwards and King~\cite[Lemma~3]{EK13}, which is proved using Hall's Theorem.
\begin{lemma}[Edwards and King \cite{EK13}]\label{halls}
  If $H$ is a graph with demand function $g$ and fractional list-assignment $L$ such that for each $S\subseteq V(H)$,
  \begin{equation*}
    \sum_{v\in S} g(v) \leq \meas{\bigcup_{v\in S}L(v)},
  \end{equation*}
  then $H$ has a $(g, L)$-coloring.
\end{lemma}

Using Lemma~\ref{halls}, we prove the following lemma, which may be of independent interest.
\begin{lemma}\label{coloring K_n minus a matching}
  Let $H \cong K_n - M$ where $M$ is a matching, and let $g$ be a demand function for $H$.
  If $L$ is a fractional list-assignment for $H$ such that
  \begin{enumerate}[(i)]
  \item for each $v\in V(H)\setminus V(M)$, we have $\meas{L(v)} \geq \sum_{u\in V(H)\setminus V(M)}g(u) + \sum_{uw\in M}\max\{g(u), g(w)\}$,
  \item for each $v\in V(M)$, we have $\meas{L(v)} \geq g(v) + \sum_{uw\in M, v\notin \{u,w\}}\max\{g(u), g(w)\}$, and 
  \item for each $uv\in M$, we have $\meas{L(u)} + \meas{L(v)} \geq \sum_{w\in V(H)}g(w)$,
  \end{enumerate}
  then $H$ has a fractional $(g, L)$-coloring. 
\end{lemma}
\begin{proof}
  Suppose not.
  Choose $H, g$, and $L$ such that $H$ has no fractional $(g, L)$-coloring and the number of edges $uv\in M$ such that $\meas{L(u)\cap L(v)} \neq \varnothing$ is minimum, and subject to that, the number of vertices $u\in V(M)$ such that $g(u) = 0$ is maximum.

  First, suppose $\meas{L(u)\cap L(v)} = \varnothing$ for each $uv\in M$.  By Lemma~\ref{halls}, there exists $S\subseteq V(H)$ such that $\sum_{v\in S}g(v) < \meas{\cup_{v\in S}L(v)}$.  By (i), if $S\setminus V(M)\neq\varnothing$, then there exists $uw\in M$ such that $\{u, w\}\subseteq S$, and by (ii), if $S\cap V(M) \neq \varnothing$, then there exists $uw\in M$ such that $\{u, w\}\subseteq S$.  Hence, there exists $uw\in M$ such that $\{u, w\}\subseteq S$.  Since $\meas{L(u)\cap L(v)} = \varnothing$, we have $\meas{\cup_{v\in S}L(v)} \geq \meas{L(u)} + \meas{L(w)}$.  Therefore by (iii), $\meas{\cup_{v\in S}L(v)} \geq \sum_{v\in V(H)}g(v)$, contradicting that $\meas{\cup_{v\in S}L(v)} < \sum_{v\in S}g(v)$.

  Therefore we may assume there exists $xy\in M$ such that $\meas{L(x)\cap L(y)}\neq\varnothing$.  We may assume without loss of generality that $g(x) \leq g(y)$.
  Let $C$ be a maximal subset of $L(x)\cap L(y)$ of measure at most $g(x)$.  Now
  \begin{itemize}
  \item for each $v\in\{x,y\}$, let $g'(v) = g(v) - \meas{C}$ and $L'(v) = L(v)\setminus C$, and 
  \item for each $v\in V(H)\setminus \{x, y\}$, let $g'(v) = g(v)$ and $L'(v) = L(v)\setminus C$.
  \end{itemize}
  Note that either $L'(x)\cap L'(y) = \varnothing$ or $g'(x) = 0$.
  Now we claim that $H$ has a fractional $(g', L')$-coloring.
  By (i), for each $v\in V(H)\setminus V(M)$,
  \begin{equation*}
    \meas{L'(v)} \geq \meas{L(v)} - \meas{C} \geq \sum_{u\in V(H)\setminus V(M)}g'(u) + \sum_{uw\in M}\max\{g'(u), g'(w)\}.
  \end{equation*}
  By (ii), for each $v\in V(M)$,
  \begin{equation*}
    \meas{L'(v)} \geq \meas{L(v)} - \meas{C} \geq g'(v) + \sum_{uw\in M, v\notin \{u,w\}}\max\{g'(u), g'(w)\}.
  \end{equation*}
  By (iii), for each $uv\in M$,
  \begin{equation*}
    \meas{L'(u)} + \meas{L'(v)} \geq \meas{L(u)} + \meas{L(v)} - 2\meas{C} \geq \sum_{w\in V(H)}g'(w).
  \end{equation*}
  By the choice of $H, g$ and $L$, the graph $H$ has a fractional $(g', L')$-coloring, as claimed, contradicting that $H$ has no fractional $(g, L)$-coloring.
\end{proof}

\section{Local Fractional Greedy Coloring}\label{caro wei proof section}

In this section we present three different proofs of Theorem~\ref{caro wei theorem}.  The first proof uses Definition~\ref{fractional coloring definition}, and it is inspired by Wei's~\cite{W81} original proof of the Caro-Wei Theorem.  The proof is also suggestive of our approach in~\cite{KP19-brooks}.
\begin{proof}[First Proof of Theorem~\ref{caro wei theorem}]
  Suppose $G$ is a minimum counterexample to Theorem~\ref{caro wei theorem}.  Let $v\in V(G)$ have minimum degree.  Since $f$ satisfies the hypothesis for $G - v$ and $G$ is a minimum counterexample, $G - v$ has an $f$-coloring.  Note that for each $u\in N(v)$, we have $f(u) \leq {1}/(d(v) + 1)$.  Therefore $v$ sees at most ${d(v)}/(d(v) + 1)$ color, so $\meas{L_\phi(v)} \geq f(v)$.  Thus, $G$ has an $f$-coloring, contradicting that $G$ is a counterexample.
\end{proof}

The second proof of Theorem~\ref{caro wei theorem} is due to Alon and Spencer~\cite{AS00}.
\begin{proof}[Second Proof of Theorem~\ref{caro wei theorem}]
  Let $I$ be a random independent set of $G$ selected according to the following distribution.  Choose a total ordering $\prec$ of $V(G)$ uniformly at random, and let $v\in I$ if $v\prec u$ for all $u\in N(v)$.  Note that each vertex is in $I$ with probability ${1}/(d(v) + 1)$.  Therefore by Proposition~\ref{equivalent definitions}\ref{equiv-definitions-prob-dist}, $G$ has an $f$-coloring, as desired.
\end{proof}

The last proof of Theorem~\ref{caro wei theorem} that we present uses the concept of a fractional list-assignment as discussed in Section~\ref{fractional list section}.  It is inspired by a proof of the Caro-Wei Theorem due to Griggs~\cite{G83}.
\begin{proof}[Third Proof of Theorem~\ref{caro wei theorem}]
  Let $G$ be a minimum counterexample, that is a graph with the fewest number of vertices having no $f$-coloring where $f$ is a demand function satisfying $f(v) \leq 1/(d(v) + 1)$ for each $v$.  Let $v\in V(G)$ such that $f(v)$ is minimum, and let $\phi(v) \subseteq [0, 1]$ have measure at least $f(v)$.  Let $f'$ be the demand function for $G - v$ such that for each $u\in V(G - v)$, we have $f'(u) = 1 / (d_{G - v}(u) + 1)$.  Since $G$ is a minimum counterexample, $G - v$ has an $f'$-coloring.  Note that for each $u\in N(v)$,
  \begin{equation*}
    f'(u)\meas{L_{\phi}(v)} = \frac{1 - f(v)}{d_{G - v}(u) + 1} \geq \frac{1 - f(u)}{d_G(u)} \geq \frac{1 - 1/(d_G(u) + 1)}{d_G(u)} = \frac{1}{d_G(u) + 1} \geq f(u).
  \end{equation*}
  Therefore by Lemma~\ref{list coloring lemma}, $G - v$ has an $(f, L_\phi)$-coloring, contradicting Proposition~\ref{extending partial coloring}.
\end{proof}

\section{Triangle-free graphs and graphs with small cliques}\label{tri-free section}

In this section, we discuss fractionally coloring triangle-free graphs and graphs with small clique number.  In Section~\ref{tri-free chi_f section}, we show that Conjecture~\ref{local tri-free} implies a recent conjecture of Cames van Batenburg et al.~\cite{CvBdJdVKP18} on the fractional chromatic number of triangle-free graphs.  In Section~\ref{tri-free approximate section}, we prove Theorems~\ref{approximate tri-free},~\ref{approximate Kr}, and~\ref{local low omega}.  We actually prove Theorem~\ref{approximate local demands version}, a more general result that can function as a blackbox to obtain a result about fractional coloring with local demands using a bound on the chromatic number, and we use this result in conjunction with the recent results of Molloy~\cite{M17}  to prove Theorems~\ref{approximate tri-free} and~\ref{approximate Kr}, respectively.  We use Theorem~\ref{approximate local demands version} in conjunction with the recent bound of Bonamy et al.~\cite[Theorem 1.6]{BKNP18} to obtain Theorem~\ref{local low omega}.  We refine this approach in Section~\ref{fractional color degree section} and reduce Conjecture~\ref{local tri-free} to Conjecture~\ref{local color degree conjecture}, a ``list-local'' version of Molloy`s result on coloring triangle-free graphs involving ``color degrees.''

\subsection{The fractional chromatic number of triangle-free graphs}\label{tri-free chi_f section}

Beyond being independently interesting, Conjecture~\ref{local tri-free} also has theoretical applications.  Recently, Cames van Batenburg et al.~\cite[Conjecture 4.3]{CvBdJdVKP18} conjectured the following.

\begin{conjecture}[Cames van Batenburg et al.~\cite{CvBdJdVKP18}]\label{tri-free chi_f conjecture}
  If $G$ is a triangle-free graph on $n$ vertices, then $\chi_f(G) \leq (\sqrt{2} + o(1))\sqrt{n/\log n}$.
\end{conjecture}

The fact that the Ramsey number $R(3, k)$ is at most $(1 + o(1))k^2/\log k$ is equivalent to the following: if $G$ is a triangle-free graph on $n$ vertices, then $n / \alpha(G) \leq (\sqrt 2 + o(1))\sqrt{n / \log n}$.  Conjecture~\ref{tri-free chi_f conjecture} is an attempt to generalize this bound to fractional coloring.  In the same way that Shearer's~\cite{Sh95} bound on the independence number implies the bound on $R(3, k)$, Conjecture~\ref{local tri-free}, if true implies Conjecture~\ref{tri-free chi_f conjecture}, as follows.

\begin{proposition}\label{local tri-free implication proposition}
  For every $\varepsilon, c > 0$, the following holds for sufficiently large $n$.  Let $G$ be a triangle-free graph on $n$ vertices with demand function $f$ such that $f(v) \geq c\log d(v)/d(v)$ for each $v\in V(G)$.  If $G$ has an $f$-coloring, then $\chi_f(G) \leq (\sqrt {2/c} + \varepsilon)\sqrt{n/\log n}$.
\end{proposition}
\begin{proof}
  We assume without loss of generality that every vertex of $G$ has degree greater than $(\sqrt {2/c} + \varepsilon)\sqrt{n/\log n} - 1$ because if $G$ has a vertex $v$ of degree at most this quantity, then it would suffice to consider $G - v$ instead, since $\chi_f(G) \leq \max\{\chi_f(G - v), (\sqrt {2/c} + \varepsilon)\sqrt{n/\log n}\}$.  In particular, every vertex of $G$ has degree sufficiently large with respect to $\varepsilon$ and $c$.
  
  Let $g$ be the demand function for $G$ where $g(v) = d(v)/n$ for each $v\in V(G)$.  Since $G$ is triangle-free, it has a $g$-coloring $\phi_1$, by assigning all vertices in the neighborhood of each vertex an interval of measure $1/n$ that is disjoint from the others.  If $G$ has an $f$-coloring $\phi_2$, then by combining the two colorings (more precisely, by letting $\phi(v) = \{x \in [0, 1/2) : 2x \in \phi_1(v)\} \cup \{x \in [1/2, 1) : 2x - 1 \in \phi_2(v)\}$), we obtain a fractional coloring $\phi$ of $G$ such that each vertex $v$ receives at least
  \begin{equation*}
    \frac{1}{2}\left(\frac{c\log d(v)}{d(v)} + \frac{d(v)}{n}\right)
  \end{equation*}
  color.  

  We claim that every $v \in V(G)$ satisfies
  \begin{equation*}
    \frac{1}{2}\left(\frac{c\log d(v)}{d(v)} + \frac{d(v)}{n}\right) \geq \left(\frac{1}{\sqrt{{2}/{c}} + \varepsilon}\right)\sqrt{\frac{\log n}{n}},
  \end{equation*}
  which completes the proof.  
  Note that   
  \begin{equation*}
    \frac{\partial}{\partial d}\left(\frac{c\log d}{d} + \frac{d}{n}\right) = \frac{c(1 - \log d)}{d^2} + \frac{1}{n}
  \end{equation*}
  and that the function $d \mapsto d^2 / (\log(d) - 1)$ is continuous and increasing for $d > e^{3/2}$.
  Hence, since $n$ is sufficiently large with respect to $c$, there is a unique number $x$ in the interval $(e^{3/2}, n)$ satisfying $x^2/(\log(x) - 1) = cn$, and each $v\in V(G)$ satisfies
  \begin{equation*}
    \frac{1}{2}\left(\frac{c\log d(v)}{d(v)} + \frac{d(v)}{n}\right) \geq \frac{1}{2}\left(\frac{c\log x}{x} + \frac{x}{n}\right).
  \end{equation*}
  From here, the intuition is that $x \approx \sqrt{cn\log n / 2}$, so the right side of the inequality above is roughly
  \begin{multline*}
    \frac{1}{2}\left(\frac{c\log\sqrt{cn\log n / 2}}{\sqrt{cn \log n / 2}} + \frac{\sqrt{cn \log n / 2}}{n}\right) \approx \frac{1}{2}\left(\frac{\sqrt{c}\log n/2}{\sqrt{n\log n / 2}} + \sqrt{\frac{c\log n}{2 n}}\right)\\ = \frac{1}{2}\left(\sqrt{\frac{c\log n}{2n}} + \sqrt{\frac{c\log n}{2n}}\right) = \sqrt{\frac{c\log n}{2n}},
  \end{multline*}
  as required.  More precisely, let $x_1 = \sqrt{cn\log n / 2}$, let $x_2 = (1 + \varepsilon\sqrt{c / 2})\sqrt{cn\log n / 2}$, and note that $x_1 < x < x_2$.  Since the function $d \mapsto c \log d / d$ is decreasing and the function $d \mapsto d / n$ is increasing, we have
  \begin{multline*}
    \frac{c\log x}{x} + \frac{x}{n} > \frac{c \log x_2}{x_2} + \frac{x_1}{n} > \frac{c \log(\sqrt{cn\log n/2})}{(1 + \varepsilon\sqrt{c/2})\sqrt{cn\log n/2}} + \sqrt{\frac{c\log n}{2n}}\\
    > \left(\frac{1}{1 + \varepsilon \sqrt{c / 2}}\right)\sqrt{\frac{c}{2}}\frac{\log(cn\log n/2)}{\sqrt{n\log n}} + \sqrt{\frac{c\log n}{2n}} > \left(\frac{2}{1 + \varepsilon \sqrt{ c / 2}}\right)\sqrt{\frac{c \log n}{2n}} = \left(\frac{2}{\sqrt{2/c} + \varepsilon}\right)\sqrt{\frac{\log n}{n}},
  \end{multline*}
  and the result follows.
\end{proof}

Cames van Batenburg et al.~\cite{CvBdJdVKP18} proved a weaker form of their conjecture with the $\sqrt 2$ replaced with a 2.  Using Proposition~\ref{local tri-free implication proposition}, this result can be improved by proving a weaker form of Conjecture~\ref{local tri-free} in which the demands for each vertex are within a factor less than two of the conjectured value.

Conjecture~\ref{local tri-free} appears very similar to the following conjecture of Harris~\cite[Conjecture~6.2]{H19}, although they are incomparable.  A graph is \textit{$d$-degenerate} if every subgraph contains a vertex of degree at most $d$.
\begin{conjecture}[Harris~\cite{H19}]\label{harris tri-free}
  If $G$ is a $d$-degenerate triangle-free graph, then $\chi_f(G) = O(d/\log d)$.
\end{conjecture}

Note that if $G$ is $d$-degenerate, then $\mad(G) \leq 2d$, where $\mad(G)$ is the maximum taken over average degrees of all subgraphs of $G$, and if $\mad(G) \leq d$, then $G$ is $d$-degenerate.  Therefore Conjecture~\ref{harris tri-free} is equivalent to the following: every triangle-free graph $G$ satisfies $\chi_f(G) = O(\mad(G) / \log\mad(G))$.  
Thus, Conjecture~\ref{harris tri-free} is the fractional coloring analogue of Shearer's~\cite{Sh83} bound on the independence number of triangle-free graphs in terms of the average degree, while Conjecture~\ref{local tri-free} is the fractional coloring analogue of Shearer's~\cite{Sh91} bound on the independence number of triangle-free graphs in terms of the degree sequence.  Although the later result of Shearer~\cite{Sh91} implies the earlier one~\cite{Sh83} using Jensen's Inequality, this implication does not hold in the fractional coloring setting.  Nevertheless, we believe progress towards one of these two conjectures should provide insight into the other.

Conjecture~\ref{harris tri-free} does not hold for the chromatic number.  As shown by Kostochka and Ne\v{s}et\v{r}il~\cite{KN99}, Tutte~\cite{D54} (under the pseudonym Blanche Descartes) provided a construction of $d$-degenerate triangle-free graphs with chromatic number $d + 1$.

Esperet, Kang, and {Thomass{\'e} \cite{EKT18} recently conjectured that every triangle-free graph with minimum degree $d$ contains a bipartite induced subgraph of minimum degree at least $\Omega(\log d)$, and they showed that Harris' conjecture, if true, implies their conjecture.
  
\subsection{Approximate versions of Conjectures~\ref{local tri-free} and~\ref{local K_r-free} and Theorem~\ref{local low omega}}\label{tri-free approximate section}

In this section, we prove Theorems~\ref{approximate tri-free},~\ref{approximate Kr}, and~\ref{local low omega} using the recent results of Molloy~\cite{M17}} and Bonamy et al.~\cite{BKNP18}.  First, we prove Theorem~\ref{approximate local demands version}, which we then use as a ``black box.''

We prove this theorem by partitioning the vertices of $G$ and coloring each part in turn using the assumption that $\chi(H) \leq \Delta / r(\Delta)$.  To that end, we introduce the following definition.

\begin{definition}
  Let $H$ be a graph with demand function $f$, and let $\varepsilon > 0$.  We say $H$ is a \textit{ladder with rungs $R_0, \dots, R_k$} with respect to $f$ and $\varepsilon$ if $(R_0, \dots, R_k)$ is a partition of $V(H)$ such that for every $i\in\{0, \dots, k - 1\}$, if $v\in R_i$ and $u\in R_j$ for $j > i$, then $f(u) \leq \varepsilon / d(v)$.  
\end{definition}

If $H$ is a ladder with rungs $R_0, \dots, R_k$, then we refer to $R_0$ as the \textit{bottom rung} and $R_k$ as the \textit{top rung}.  In our proofs, we choose the rungs based on the degrees of vertices.  Essentially, the bottom rung will contain the vertices of minimum degree, and the top rung will contain the vertices of maximum degree.  

\begin{lemma}\label{ladder coloring lemma}
  If $H$ is a ladder with rungs $R_0, \dots, R_k$ with respect to a demand function $f$ and $\varepsilon$  such that $H[R_i]$ has an $f/(1 - \varepsilon)$-coloring for each $i\in\{0, \dots, k\}$, then $H$ has an $f$-coloring.
\end{lemma}

\begin{proof}
Let $\phi$ be an $f$-coloring of $H[S]$ for $S\subseteq V(H)$ maximal subject to the following: If $R_i\cap S \neq \varnothing$, then $R_i\subseteq S$, and if $R_i\subseteq S$ for $i < k$, then $R_{i+1} \subseteq S$.  That is, $\phi$ is an $f$-coloring of as many rungs of $H$ as possible, from top to bottom.  The empty set satisfies this condition, so such an $S$ exists.  Suppose for a contradiction that $S\neq V(H)$.  Let $i$ be maximum such that $R_i\cap S = \varnothing$, and let $v\in R_i$.  Since $H$ is a ladder with respect to $f$ and $\varepsilon$, the vertex $v$ sees at most $\sum_{u\in N(v)\cap \cup_{j=i+1}^k R_j}f(u) \leq \varepsilon$ color from $\phi$.  By assumption, $H[R_i]$ has a fractional coloring in which each vertex receives at least $f(v)/(1 - \varepsilon)$ color, so by Lemma~\ref{list coloring lemma}, $\phi$ can be extended to an $f$-coloring of $H[R_i]$, contradicting the choice of $S$.  Therefore $H$ has an $f$-coloring, as desired.
\end{proof}

In the proof of Theorem~\ref{approximate local demands version}, we partition the vertices of $G$ into two ladders and color each with half of the available color.  We apply Lemma~\ref{ladder coloring lemma} to both of these ladders to obtain this coloring.  In the next section, we also use Lemma~\ref{ladder coloring lemma} in a similar way.

In order to apply Lemma~\ref{ladder coloring lemma}, we need to show that each rung in one of our ladders has the desired coloring.  We need the following definition.

\begin{definition}
  Let $H\subseteq G$ be a ladder with rungs $R_0, \dots, R_k$.  A rung $R$ is \textit{$(r, M)$-stratified by $S_1, \dots, S_M$} if $(S_1, \dots, S_M)$ is a partition of $R$ such that for every $i\in\{1, \dots, M\}$, if $u, v\in S_i$, then $d_G(u) \leq d_G(v)\cdot r(d_G(v))^{1/M}$.
\end{definition}

We use the following lemma to color each rung of the ladders in the proof of Theorem~\ref{approximate local demands version}.

\begin{lemma}\label{coloring stratified lemma}
  Suppose a rung $R$ of a ladder $H\subseteq G$ is $(r, M)$-stratified by $S_1, \dots, S_M$ such that for every $i\in\{1, \dots, M\}$ and every $\Delta$, if $\max_{v\in S_i}d_G(v) \leq \Delta$, then $\chi_f(G[S_i]) \leq \Delta/r(\Delta)$.  If $g : \mathbb R\rightarrow \mathbb R$ satisfies
  \begin{equation}\label{gets enough color}
    g(x) \leq \frac{r(x)}{xM\cdot r(x)^{1/M}}
  \end{equation}
  and $f$ is a demand function for $G$ where $f(v) = g(d(v))$ for each vertex $v$, then the graph $G[R]$ has an $f$-coloring.
\end{lemma}
\begin{proof}
We color each graph $G[S_i]$ separately with disjoint sets of color of measure $1/M$.  For each $i\in\{1, \dots, M\}$, let $\Delta_i = \max_{v\in S_i}d_G(v)$.  By assumption, for each $i \in \{1, \dots, M\}$, we have $\chi_f(G[S_i]) \leq \Delta_i/ r(\Delta_i)$.  Therefore for each $i\in\{1, \dots, M\}$ there is a fractional coloring $\phi_i$ of $G[S_i]$ such that each vertex in $S_i$ receives at least $r(\Delta_i)/(M\Delta_i)$ color and the range of $\phi_i$ and $\phi_j$ is disjoint if $i\neq j$.  Since $R$ is $(r, M)$-stratified, for each $i\in \{1, \dots, M\}$ and $v\in S_i$, we have $\Delta_i \leq d(v)\cdot r(d(v))^{1/M}$.  Since $r$ is nondecreasing, each vertex $v$ receives at least $r(d(v)) / (M d(v)\cdot r(d(v))^{1/M})$ color.  By~\eqref{gets enough color}, $v$ receives at least $g(d(v))$ color, as desired.
\end{proof}

Combining Lemmas~\ref{ladder coloring lemma} and~\ref{coloring stratified lemma}, we now prove Theorem~\ref{approximate local demands version}.

\begin{proof}[Proof of Theorem~\ref{approximate local demands version}]
  We may assume that every vertex in $G$ has degree at least $\delta^{-1} - 1$, since $f(v) \leq \delta$ for every $v\in V(G)$.  We do not explicitly compute $\delta$; instead, we assume $\delta > 0$ is sufficiently small to satisfy certain inequalities throughout the proof.  In particular, every vertex of $G$ has arbitrarily large degree.

  Let $M_0, M_1, \dots$ and $M'_0, M'_1, \dots$ be nondecreasing sequences of positive integers, and let $\delta_{i,j}$ for $i\geq0$ and $j \in \{0, \dots, M_i\}$ and $\delta'_{i,j}$ for $i \geq 0$ and $j \in \{0, \dots, M'_i\}$ be positive reals to be determined later.  We partition the vertices of $G$ into two parts that induce ladders $H$ and $H'$ with rungs $R_0, \dots, R_k$ and $R'_0, \dots, R'_{k}$ such that the rungs $R_i$ and $R'_i$ are $(r, M_i)$-stratified and $(r, M'_i)$-stratified, respectively, as follows.  Whether a vertex is in $H$ or in $H'$, and its rung in the corresponding ladder, is determined by its degree with respect to the $\delta_{i,j}$ and $\delta'_{i,j}$.

  Let $\delta_{0, 0}$ be the minimum degree of a vertex in $G$, and define $\delta'_{i, j}$ for $i\geq 0$ and $j \in \{0, \dots, M'_i\}$ and $\delta_{i, j}$ for $i\geq 0$ and $j \in \{0, \dots, M_j\}$ (and $(i,j) \neq (0,0)$) as follows.  For each $i\geq 0$ and each $j\in\{1, \dots, M_i\}$, let $\delta_{i, j} = \delta_{i, j-1}\cdot r(\delta_{i, j - 1})^{1/M_i}$, and similarly, for each $i\geq 0$ and each $j\in \{1, \dots, M'_i\}$, let $\delta'_{i, j} = \delta'_{i, j-1}\cdot r(\delta_{i, j - 1})^{1/M'_i}$.  For each $i\geq 0$, let $\delta'_{i, 0} = \delta_{i, M_i}$, and for each $i\geq 1$, let $\delta_{i, 0} = \delta'_{i - 1, M'_{i - 1}}$.  Note that the constants $\delta_{i,j}$ and $\delta'_{i,j}$ are well-defined and satisfy
  \begin{equation*}
    \delta_{0,0} < \cdots < \delta_{0,M_0} = \delta'_{0,0} <  \cdots < \delta'_{0, M'_0} = \delta_{1, 0} < \cdots < \delta_{1, M_1} = \delta'_{1,0} <  \cdots < \delta'_{1, M'_1} = \delta_{2,0} <  \cdots
  \end{equation*}

  For each $i \geq 0$ and $j \in \{1, \dots, M_i\}$, let $S_{i, j}$ be the set of vertices of $G$ with degree at least $\delta_{i, j - 1}$ and less than $\delta_{i, j}$, and for each $i \geq 0$ and $j \in \{1, \dots, M'_i\}$, let $S'_{i, j}$ be the set of vertices of $G$ with degree at least $\delta'_{i, j - 1}$ and less than $\delta'_{i, j}$.  
  For each $i\geq 0$, let $R_i$ be the set of vertices of $G$ with degree at least $\delta_{i, 0}$ and less than $\delta'_{i, 0}$, and let $R'_i$ be the set of vertices of $G$ with degree at least $\delta'_{i, 0}$ and at most $\delta_{i + 1, 0}$.
  Let $H$ be the graph induced by $G$ on $\cup R_i$, and let $H'$ be the graph induced by $G$ on $\cup R'_i$.  Since $G$ is finite, we let $k$ be the largest integer such that there is a vertex in $R_k\cup R'_k$.  We may assume without loss of generality that there is a vertex in $R'_k$.

  Since $r$ is nondecreasing and greater than 1, for each $i\in\{0, \dots, k\}$, $\delta'_{i, 0} > \delta_{i, 0}\cdot r(\delta_{i, 0})$.  Similarly, for each $i\geq \{1, \dots, k\}$, $\delta_{i, 0} > \delta'_{i - 1, 0}$.  Therefore for each $i\in \{0, \dots, k\}$, the graphs $H[R_i]$ and $H[R'_i]$ are $(r, M_i)$-stratified and $(r, M'_i)$-stratified by $(S_{i, 1}, \dots, S_{i, M_i})$ and $(S'_{i, 1}, \dots, S'_{i, M'_i})$, respectively.  For each $i\in \{0, \dots, k\}$, let $g_i(x) = \frac{r(x)}{xM_i\cdot r(x)^{1/M_i}}$, and let $g'_i(x) = \frac{r(x)}{xM'_i\cdot r(x)^{1/M'_i}}$.  Let $g$ and $g'$ be demand functions for $H$ and $H'$ respectively, where $g(v) = (1 - \varepsilon/2)g_i(d(v))$ if $v\in R_i$ and $g'(v) = (1 - \varepsilon/2)g'_i(d(v))$ if $v\in R'_i$.  We claim that $H$ is a ladder with rungs $(R_0, \dots, R_k)$ with respect to $g$ and $\varepsilon/2$ and $H'$ is a ladder with rungs $(R'_0, \dots, R'_k)$ with respect to $g'$ and $\varepsilon/2$.  To that end, let $v\in R_i$.  For each $u\in \cup_{j=i+1}^k R_j$, we have
  \begin{equation*}
    g(u) \leq \frac{r(d(v)\cdot r(d(v)))}{d(v)r(d(v))M_j\cdot r(d(v)\cdot r(d(v)))^{1/M_j}}.
  \end{equation*}
  Note that for any $x$, the function $M \cdot x^{1/M}$ is maximized when $M = \log x$ and is thus at most $e\log x$.  Letting $M = M_j$ and $x = r(d(v)\cdot r(d(v)))$, we have for each $u\in R_j$ for $j > i$,
  \begin{equation*}
    g(u) \leq \frac{r(d(v)\cdot r(d(v)))}{d(v)r(d(v))e\log (r(d(v)\cdot r(d(v))))}.
  \end{equation*}  
  Therefore by \eqref{r constraint}, we may assume the right side of the previous inequality is at most $\varepsilon/(2d(v))$, as required.  Hence, $H$ is a ladder as claimed, and the proof for $H'$ is the same.

  Let $f$ be the demand function for $G$ where $f(v) = g(v)/2$ if $v\in R_i$ and $f(v) = g'(v)/2$ if $v\in R'_i$.  Note that for $v\in R_i$, since $r$ is nondecreasing, $f(v) \geq \frac{(1 - \varepsilon/2)r(d(v))}{2d(v)M_i\cdot r(\delta'_{i, 0})^{1/M_i}}$.  Letting $M_i = \lceil\log(\delta'_{i, 0})\rceil$ and $M'_i = \lceil\log(\delta_{i + 1, 0})\rceil$, assuming $\varepsilon < 1$, we have $f(v) \geq (2e + \varepsilon)^{-1}\frac{r(d(v))}{d(v)\log(r(\delta'_{i, 0}))}$.  Since $d(v)r(d(v)) \geq \delta'_{i, 0}$, we have
  \begin{equation*}
    f(v) \geq (2e + \varepsilon)^{-1}\frac{r(d(v))}{d(v)\log(r(d(v)\cdot r(d(v))))}.
  \end{equation*}
  
  Therefore it suffices to show that $G$ has an $f$-coloring.  If $H$ and $H'$ have $g$ and $g'$-colorings respectively, then $G$ has an $f$-coloring, obtained by averaging $g$ and $g'$.  Thus, we only need to show that $H$ has a $g$-coloring, since the proof that $H'$ has a $g'$-coloring is the same.  By Lemma~\ref{coloring stratified lemma}, $H[R_i]$ has a fractional coloring in which each vertex receives at least $g_i(d(v))$ color, so by Lemma~\ref{ladder coloring lemma}, $H$ has a $g$-coloring, as required. 
\end{proof}

Now we show how to obtain Theorems~\ref{approximate tri-free},~\ref{approximate Kr}, and~\ref{local low omega} using Theorem~\ref{approximate local demands version}.

\begin{proof}[Proof of Theorem~\ref{approximate tri-free}]
  For each $d\in\mathbb N$, let $r(d) = (1 + o(1))^{-1}\log d$, and note that $r$ satisfies~\eqref{r constraint} and is increasing, as required.  Since $G$ is a triangle-free graph, every subgraph $H\subseteq G$ of maximum degree at most $\Delta$ satisfies $\chi(H) \leq \Delta/r(\Delta)$ by~\cite[Theorem~1]{M17}, as required.  Since
  \begin{equation*}
    \frac{r(d)}{d\log r(d\cdot r(d))} = (1 + o(1))^{-1}\frac{\log d}{d\log\log d},
  \end{equation*}
  the result follows from Theorem~\ref{approximate local demands version}.
\end{proof}

\begin{proof}[Proof of Theorem~\ref{approximate Kr}]
  For each $d\in\mathbb N$, let $r(d) = \log d / (200\omega \log\log d)$, and note that $r$ satisfies~\eqref{r constraint} and is increasing, as required.  Since $G$ has clique number less than $r$, every subgraph $H\subseteq G$ of maximum degree at most $\Delta$ satisfies $\chi(H) \leq \Delta/r(\Delta)$ by~\cite[Theorem~2]{M17}, as required.  Since
  \begin{equation*}
    \frac{r(d)}{d\log r(d\cdot r(d))} = O\left(\frac{\log d}{d(\log\log d)^2}\right),
  \end{equation*}
  the result follows from Theorem~\ref{approximate local demands version}.
\end{proof}

\begin{proof}[Proof of Theorem~\ref{local low omega}]
  For each $d\in\mathbb N$, let $r(d) = 6c\log c$, and note that $r$ satisfies~\eqref{r constraint} with $\varepsilon = 1/(6e)$ and is nondecreasing, as required.  Since $\omega(v) \leq d(v)^{1/(432c\log c)^2}$ for each vertex $v$, every subgraph $H\subseteq G$ such that $\max_{v\in V(H)}d_G(v) \leq \Delta$ satisfies $\omega(H) \leq \Delta^{1/(432c\log c)^2}$, and thus by~\cite[Theorem~1.6]{BKNP18}, $\chi(H) \leq \Delta/r(\Delta)$, as required.  Since $c \geq e^6$, we have
  \begin{equation*}
    \left(2e + \frac{1}{6e}\right)^{-1}\frac{r(d)}{d\log r(d\cdot r(d))} = \left(\frac{6}{2e + 1/(6e)}\right)\frac{c\log c}{d(\log 6 + \log c - \log\log c)} \geq \frac{c}{d},
  \end{equation*}
  and thus the result follows from Theorem~\ref{approximate local demands version}.
\end{proof}

\subsection{Color degrees}\label{fractional color degree section}

In this section, we reduce Conjecture~\ref{local tri-free} to a problem involving list coloring and \textit{color degrees}.  We need the following definition.

\begin{definition}
  Let $G$ be a graph with list-assignment $L$.
  \begin{itemize}
  \item For each $v\in V(G)$ and $c\in L(v)$, the \textit{color degree} of $v$ and $c$, denoted $d_{G, L}(v, c)$, is the number of neighbors $u\in N(v)$ such that $c\in L(u)$.
  \item We let $\Delta(G, L)$ denote the \textit{maximum color degree}, the maximum over $v\in V(G)$ and $c\in L(v)$ of $d(v, c)$.
  \end{itemize}
\end{definition}

In 1999, Reed~\cite{R99-2} conjectured that if $G$ is a graph with list-assignment $L$ such that for each $v\in V(G)$, the number $|L(v)|$ of available colors for $v$ is at least $\Delta(G, L) + 1$, then $G$ is $L$-colorable.  In some sense, this conjecture is a ``color degree'' analogue of the greedy bound on the chromatic number.  Although, Bohman and Holzman~\cite{BH02} disproved this conjecture, Reed and Sudakov~\cite{RS02} showed instead that if the number $|L(v)|$ of available colors is at least $(1 + o(1))\Delta(G, L)$ for each $v$, then the graph is $L$-colorable.

Before Johansson~\cite{J96-tri} proved that triangle-free graphs have chromatic number $O(\Delta / \log \Delta)$, Kim~\cite{K95} proved that graphs of girth at least five have list chromatic number at most $(1 + o(1))\Delta / \log\Delta$.  Molloy~\cite{M17} strengthened this result by showing that it also holds for triangle-free graphs; however, Molloy and Reed~\cite{MR02} provided a simpler proof of Kim's result that actually holds when $\Delta$ is replaced by the maximum color degree.  Pettie and Su~\cite{PS15} proved a $(4 + o(1))\Delta / \log\Delta$-bound on the list chromatic number of triangle-free graphs, and their proof also holds when $\Delta$ is replaced by $\Delta(G, L)$.  Molloy's proof does not work for color degrees, although we believe the bound is true.  In fact, we make the following more general conjecture.

\begin{conjecture}\label{local color degree conjecture}
  For every $\varepsilon > 0$, there exists $\alpha > 0$ and $\Delta$ sufficiently large such that the following holds.  If $G$ is a triangle-free graph with list-assignment $L$ such that each $v\in V(G)$ satisfies $|L(v)| \geq \Delta^{1-\alpha}$ and for each $c\in L(v)$ we have
  \begin{equation*}
    |L(v)| \geq (1 + \varepsilon)\frac{d(v, c)}{\log d(v, c)}
  \end{equation*}
  and $d(v, c) \leq \Delta$, then $G$ is $L$-colorable.
\end{conjecture}

The main result of this subsection is that Conjecture~\ref{local color degree conjecture}, if true, implies Conjecture~\ref{local tri-free}.

Conjecture~\ref{local color degree conjecture} generalizes Molloy's~\cite{M17} bound on the list chromatic number of triangle-free graphs in two ways: first, it only requires the number of available colors for each vertex to depend on a local parameter, rather than a global one (within the range $[\Delta^{1-\alpha}, \Delta]$), and second, it replaces degrees with color degrees.  Bonamy et al.~\cite{BKNP18} and Davies et al.~\cite{DdJdVKP18} proved variations of Molloy's result of the first type.  As mentioned, Molloy and Reed's~\cite{MR02} proof of Kim's~\cite{K95} bound for girth five graphs and Pettie and Su's~\cite{PS15} bound for triangle-free graphs can be adapted to be of the second type.
A result of Davies et al.~\cite[Proposition~11]{DdJdVKP18} implies that some lower bound on either the color degrees or list sizes in terms of $\Delta$ in Conjecture~\ref{local color degree conjecture} is necessary; however, we believe that the $\Delta^{1-\alpha}$-bound could be lowered to $\mathrm{poly}\log\Delta$.

In the remainder of the section, we show that Conjecture~\ref{local color degree conjecture}, if true, implies Conjecture~\ref{local tri-free}.  The following definition is crucial.

\begin{definition}
  Let $G$ be a graph with demand function $f$, and let $N$ be a common denominator for $f$ and $f/(1 + f)$.  If $G'$ is a graph obtained from $G$ by replacing each vertex $v$ with an independent set $\{v_1, \dots, v_{f(v)\cdot N/(1 + f(v))}\}$ and if $L$ is a list-assignment for $G'$ such that the lists $L(v_i)$ partition $\{1, \dots, N\}$ and are each of size at least $\lceil f(v)^{-1}\rceil$, then $(G, L)$ is an \textit{$N$-color-partitioned blowup} of $G$ with respect to $f$ and the vertex $v$ is the \textit{progenitor} of the vertices in $\{v_1, \dots, v_{f(v)\cdot N/(1 + f(v))}\}$.
\end{definition}

Note that in this definition, every vertex $v\in V(G)$ satisfies $\lceil f(v)^{-1}\rceil \cdot f(v) \cdot N/(1 + f(v)) \leq N$.  Hence, $N$-color-partitioned blowups of $G$ indeed exist for every such $N$.

The following proposition reveals the connection between color degrees and fractional coloring, using the notion of the $N$-color-partitioned blowup.

\begin{proposition}\label{blowup color degree prop}
  Let $G$ be a graph with demand function $f$, and let $(G', L)$ be an $N$-color-partitioned blowup of $G$ with respect to $f$.
  \begin{enumerate}[(a)]
  \item\label{blowup color degree L-col} If $G'$ is $L$-colorable, then $G$ has an $f/(1 + f)$-coloring, and
  \item\label{blowup color degree progenitor} for each $v' \in V(G')$ and $c\in L(v)$, if $v\in V(G)$ is the progenitor of $v'$, then $d(v', c) = d(v)$.
  \end{enumerate}
\end{proposition}
\begin{proof}
  Let $\phi$ be an $L$-coloring of $G'$.  For each $v\in V(G)$, let
  \begin{equation*}
    \psi(v) = \cup_{i\in\{1, \dots, f(v)\cdot N/(1 + f(v))\}}\phi(v_i),
  \end{equation*}
  where $v$ is the progenitor for each vertex in $\{v_1, \dots, v_{f(v)\cdot N/(1 + f(v))}\}$.  Since the lists $L(v_i)$ partition $\{1, \dots, N\}$, the colors $\phi(v_i)$ are distinct.  Hence, $|\psi(v)| = f(v)\cdot N/(1 + f(v))$.  Moreover, $\psi(v)\cap \psi(u) = \varnothing$ for every $uv\in E(G)$.  Therefore $\psi$ is an $(f/(1 + f), N)$-coloring of $G$.  By Theorem~\ref{equivalent definitions}, $G$ has an $f/(1 + f)$-coloring, as desired.

  Since the lists $L(v_i)$ partition $\{1, \dots, N\}$, for every $c\in L(v)$ and $u\in N(v)$, there is precisely one vertex $u'$ of which $u$ is the progenitor such that $c\in L(u')$.  Therefore $d(v', c) = d(v)$, as claimed.
\end{proof}

Let $r : \mathbb N \rightarrow \mathbb R$ and $a, b \in \mathbb R$.  We say a graph $G$ is \textit{$r$-locally color-degree list-colorable for color-degrees in $[a,b]$} if $G$ is $L$-colorable for any list-assignment $L$ such that each $v\in V(G)$ and $c\in L(v)$ satisfies
\begin{equation*}
  |L(v)| \geq d(v, c) / r(d(v, c))
\end{equation*}
and $a \leq d(v, c) \leq b$.  A class of graphs $\mathcal G$ is \textit{closed under duplicating vertices} if for every $G\in\mathcal G$ and $v\in V(G)$, the graph obtained from $G$ by adding a vertex adjacent to each vertex in $N(v)$ is in $\mathcal G$.  Note that if $\mathcal G$ is closed under duplicating vertices, $G\in\mathcal G$, and $(G', L)$ is an $N$-color-partitioned blowup of $G$, then $G'\in\mathcal G$.

The following theorem is the main result of this subsection.
\begin{theorem}\label{color degree reduction theorem}
  Let $r : \mathbb R \rightarrow \mathbb R$ be increasing and tending to infinity such that for any $\varepsilon > 0$, if $x$ is sufficiently large, then $r(3\varepsilon^{-1}x\cdot r(x)) \leq (1 - \varepsilon)3r(x)/2$ and $r(x)^{1/\varepsilon} = o(x^\alpha)$ for every $\alpha > 0$.
  Let $\mathcal G$ be a hereditary class of graphs closed under duplicating vertices such that for some $\alpha > 0$ and every sufficiently large $\delta$, every graph $G\in \mathcal G$ is $r$-locally color-degree list-colorable for color-degrees in $[\Delta^{1- \alpha}, \Delta]$.

  For every $\varepsilon > 0$, there exists $\delta > 0$ such that the following holds.  If $G\in \mathcal G$ has demand function $f$ such that
  \begin{equation*}
    f(v) \leq \min\left\{(1 + \varepsilon)^{-1}\frac{r(d(v))}{d(v)}, \delta\right\}
  \end{equation*}
  for each $v\in V(G)$, then $G$ has an $f$-coloring.
\end{theorem}
\begin{proof}
  It suffices to prove the result for $\varepsilon = 1 / n$ for $n\in\mathbb N $ sufficiently large.  We may assume that every vertex has degree at least $\delta^{-1} - 1$, since $f(v) \leq \delta$ for every $v\in V(G)$.  We choose $n$ sufficiently large and $\delta$ sufficiently small so that every vertex in $G$ has large enough degree to satisfy certain inequalities throughout the proof.  We separate the vertices of $G$ into $n$ different ladders such that each vertex is contained in $n - 1$ of them, as follows.
  
  Let $\delta_{0, 0} = \delta^{-1} - 1$.  For $i\geq 0$ and $j \in \{1, \dots, n\}$, let $\delta_{i, j} = 3\varepsilon^{-1}\delta_{i, j - 1}\cdot r(\delta_{i, j - 1})$, and for $i \geq 1$, let $\delta_{i, 0} = \delta_{i - 1, n}$.  For each $i \geq 0$ and $j\in\{1, \dots, n\}$, let $R_{i, j}$ be the set of vertices of $G$ with degree at least $\delta_{i, j}$ and less than $\delta_{i + 1, j - 1}$.  For each $j\in \{1, \dots, n\}$, let $H_j$ be the graph induced by $G$ on $\cup_i R_{i, j}$, and for each $i \geq 0$, let $H_{i, j}$ be the graph induced by $G$ on $R_{i, j}$.   Let $f' = (1 + \varepsilon)f$.  Using Lemma~\ref{ladder coloring lemma}, we will show that each graph $H_j$ has an $f'$-coloring.

  To that end, we claim that $H_\ell$ is a ladder with rungs $(R_{i, \ell})_{i\geq 0}$ with respect to $(1 + \varepsilon)f$ and $(1 - \varepsilon)\varepsilon/2$.  For each $v\in R_{j, \ell}$ and $u \in \cup_{i\geq j + 1}R_{i, \ell}$, since $r$ is increasing, by the choice of the $\delta_{i,j}$ and $R_{i,j}$ we have $d(u) > 3\varepsilon^{-1}d(v)\cdot r(d(v))$ and thus
  \begin{equation*}
    (1 + \varepsilon) f(u) \leq \frac{r(3\varepsilon^{-1}d(v)\cdot r(d(v)))}{3\varepsilon^{-1}d(v)\cdot r(d(v))}.
  \end{equation*}
  Since $r(3\varepsilon^{-1}x\cdot r(x)) \leq (1 - \varepsilon)3r(x)/2$, the right side of the previous inequality is at most $\varepsilon(1 - \varepsilon)/(2d(v))$.  Therefore $(1 + \varepsilon)f(u)/\leq \varepsilon(1 - \varepsilon)/(2d(v))$, as required.
  
  Next we will show that each $H_{i,j}$ has an $f' / (1 - (1 - \varepsilon)\varepsilon/2)$-coloring using Proposition~\ref{blowup color degree prop}.  To that end, Let $N$ be a common denominator for $f'$ and $f'/(1 + f')$, and for each $i\geq 0$ and $j\in\{1, \dots, n\}$, let $(H'_{i, j}, L_{i, j})$ be an $N$-color-partitioned blowup of $H_{i,j}$.  Since $\mathcal G$ is hereditary, each graph $H_j\in\mathcal G$, and since $\mathcal G$ is closed under taking blowups, each graph $H'_{i, j}\in\mathcal G$.  Note that each vertex $u \in V(H'_{i,j})$ with progenitor $v$ satisfies $|L(u)| \geq \lceil f'(v)^{-1}\rceil \geq d(v) / r(d(v))$. 

  We claim that for each $i\geq 0$ and each $j\in\{1, \dots, n\}$, we have that $\delta_{i+1, j - 1} \leq (3/2)^{\binom{n}{2}}(3\varepsilon^{-1}r(\delta_{i, j}))^{n - 1}\delta_{i, j}$.  We prove this claim for the case $j = n$, and the proof for other $j$ is the same.  Since $\delta_{i + 1, 0} = \delta_{i, n}$, it suffices to show that for each $i\geq 0$ we have $\delta_{i, n - 1} \leq (3/2)^{\binom{n}{2}}(3\varepsilon^{-1}r(\delta_{i, 0}))^n\delta_{i, 0}$.  We actually prove by induction the stronger statement that for each $i\geq 0$ and $\ell \in \{0, \dots, n - 1\}$, we have $\delta_{i, \ell} \leq (3/2)^{\binom{\ell+1}{2}}(3\varepsilon^{-1}r(\delta_{i, 0}))^{\ell}\delta_{i, 0}$.  Suppose this statement is true for $\ell' \leq \ell$, where $\ell \geq 0$.  By the definition of $\delta_{i, \ell+1}$ and the inductive hypothesis, we have
  \begin{equation}\label{delta inductive hypothesis}
    \delta_{i, \ell + 1} = 3\varepsilon^{-1}\delta_{i, \ell}\cdot r(\delta_{i, \ell}) \leq 3\varepsilon^{-1}(3/2)^{\binom{\ell}{2}}(3\varepsilon^{-1}r(\delta_{i, 0}))^{\ell} \delta_{i, 0}\cdot r(\delta_{i, \ell}).
  \end{equation}
  Since $r$ is increasing and $r$ satisfies $r(3\varepsilon^{-1}x\cdot r(x)) \leq 3r(x)/2$, assuming $\delta_{0, 0}^{-1}$ is sufficiently large, we have
  \begin{equation}\label{r delta induction}
    r(\delta_{i, \ell}) \leq (3/2)^\ell r(\delta_{i, 0})
  \end{equation}
  Combining~\eqref{delta inductive hypothesis} and~\eqref{r delta induction}, we have $\delta_{i, \ell + 1} \leq (3/2)^{\binom{\ell + 1}{2}}(3\varepsilon^{-1}r(\delta_{i, 0}))^\ell\cdot \delta_{i, 0}$, as required.  Since the case $\ell = 0$ trivially holds, the claim follows.

  Therefore, since $r(x)^{1/\varepsilon} = r(x)^n = o(x^\alpha)$, for $\delta^{-1}$ sufficiently large we have $\delta_{i + 1, j - 1} \leq \delta_{i, j}^{1 + \alpha}$.
  Since every graph in $\mathcal G$ is $r$-locally color-degree list-colorable for color degrees in $[\Delta^{1-\alpha}, \Delta]$, for each $j$, by Proposition~\ref{blowup color degree prop}\ref{blowup color degree progenitor}, we have that $H'_{i, j}$ is $L_{i, j}$-colorable for each $N$-color-partitioned blowup $(H'_j, L_j)$.  By Proposition~\ref{blowup color degree prop}\ref{blowup color degree L-col}, each graph $H_{i, j}$ has an $f'/(1 + f')$-coloring.  Since $f(v) \leq \delta$ for each vertex $v$, we may assume $f'/(1 + f') \geq f'/(1 + (1 + \varepsilon)\delta) \geq f'/(1 - (1 - \varepsilon)\varepsilon/2)$.  Therefore by applying Lemma~\ref{ladder coloring lemma} with $f'$ and $\varepsilon(1 - \varepsilon)/2$, each graph $H_j$ has a $(1 + \varepsilon)f$-coloring.  Since $\varepsilon = 1/n$ and each vertex is in $n - 1$ of the graphs $H_j$, averaging these $(1 + \varepsilon)f$-colorings yields an $f$-coloring of $G$, as desired.
\end{proof}

Note that the function $r(x) = \log x$ and the class $\mathcal G$ of triangle-free graphs satisfy the hypotheses of Theorem~\ref{color degree reduction theorem}.  Moreover, if for each $v\in V(G)$ and $c\in L(v)$ we have $d(v, c) \geq \Delta^{1-\alpha}$ and $|L(v)| \geq d(v, c) / \log d(v, c)$, then $|L(v)| \geq \Delta^{1 - 2\alpha}$ for sufficiently large $\Delta$.  Therefore by Theorem~\ref{color degree reduction theorem}, Conjecture~\ref{local color degree conjecture}, if true, implies Conjecture~\ref{local tri-free}.

It is possible that a variant of Conjecture~\ref{local color degree conjecture} holds for $K_r$-free graphs rather than triangle-free graphs if we increase the lower bound on the list size in terms of color degrees by a constant factor.  However, proving this is out of reach with our current techniques, even with color degrees replaced by maximum degree.  We may instead consider a list assignment $L$ in which each vertex $v$ and $c\in L(v)$ satisfies $|L(v)| = \Omega(d(v, c)\log\log d(v, c) / \log d(v, c))$.  Proving that the graph is $L$-colorable in this case would be a ``color degree'' version of Johansson's~\cite{J96-Kr} bound on the chromatic number of $K_r$-free graphs, and by Theorem~\ref{color degree reduction theorem}, it would imply a local demands version of Johansson's and Shearer's~\cite{Sh95} bounds.  Neither Johansson's nor Molloy's~\cite{M17} proof easily adapts for color degrees, though.

\section{$\chi$-boundedness}\label{chi-bounded section}

Recall that a class of graphs is $\chi$-bounded if the chromatic number of every induced subgraph of a graph in the class is bounded by a function of its clique number.  In this case, such a function is called a \textit{$\chi$-binding} function.  There is an enormous amount of research devoted to studying $\chi$-boundedness.  In the 1980s, Gy\'{a}rf\'{a}s~\cite{G87} posed several beautiful conjectures that were critical in popularizing the subject.  A \textit{hole} in a graph is an induced cycle of length at least four, and an \textit{antihole} is an induced subgraph whose complement is a hole.  The Strong Perfect Graph Theorem~\cite{CRST06} implies that graphs with no odd hole or odd antihole are $\chi$-bounded by the identity function.  At that time, it was not known if this family of graphs is $\chi$-bounded at all, so Gy\'{a}rf\'{a}s~\cite[Conjecture~3.2]{G87} proposed the ``Weakened Strong Perfect Graph Conjecture.''  He conjectured three strengthenings of this problem: that graphs with no odd holes, graphs with no long holes (that is, holes of length at least $\ell$, for any $\ell$), and graphs with no long odd holes are $\chi$-bounded.  Note that the third case includes the former two.  The first of these conjectures was recently proved by Scott and Seymour~\cite{SS16}, the second by Chudnovsky, Scott, and Seymour~\cite{CSS17}, and the third by Chudnovsky, Scott, Seymour, and Spirkl~\cite{CSSS17}.  Finally, Scott and Seymour~\cite{SS17} generalized all of these results by proving that graphs with no holes of a specific residue are $\chi$-bounded.

Gy\'{a}rf\'{a}s~\cite{G87} also asked for which graphs $H$ is the class of \textit{$H$-free} graphs, that is the graphs  with no induced subgraph isomorphic to $H$, $\chi$-bounded.  The Gy\'{a}rf\'{a}s-Sumner Conjecture~\cite{G75, S81} states that if $H$ is a tree, then the class of $H$-free graphs is $\chi$-bounded.  Note that since there exist graphs of arbitrarily large girth and chromatic number as proved by Erd\H os~\cite{E59}, it is necessary that $H$ does not contain a cycle.  The conjecture has been proved for various families of trees, but it remains open.

Finding the best possible $\chi$-binding functions for various graph classes has also been extensively studied.  In addition to classes of graphs characterized by excluded induced subgraphs, geometric graphs -- graphs described by the intersections of a collection of geometric objects -- have received much attention.  Each object corresponds to a vertex and vertices are adjacent if the corresponding objects intersect.  For example, Asplund and Gr\"{u}nbaum~\cite{AG60} proved that the family of intersection graphs of axis-parallel rectangles is $\chi$-bounded, and determining the optimum $\chi$-binding function remains an interesting open problem.  

There are many more interesting results and problems in the study of $\chi$-boundedness.  For more information, see the recent survey of Scott and Seymour~\cite{SS18}.

In this section, we introduce the concept of local-fractional $\chi$-boundedness.

\begin{definition}
  Let $\mathcal G$ be a class of graphs.
  \begin{itemize}
  \item If there exists a function $g : \mathbb N \rightarrow \mathbb R$ such that for all $G\in \mathcal G$, and for all induced subgraphs $H$ of $G$, the graph $H$ has an $f$-coloring for every demand function $f$ such that for each $v\in V(H)$, $f(v) \leq 1/g(\omega(v))$, then $\mathcal G$ is \textit{local-fractionally $\chi$-bounded};
  \item in this case, the class $\mathcal G$ is \textit{local-fractionally $\chi$-bounded} by $g$ and $g$ is a \textit{local-fractional $\chi$-binding function} for $\mathcal G$.
  \end{itemize}
\end{definition}

The following is the main result of this section.  In this section, a \textit{blowup} of a graph $G$ is a graph obtained from $G$ by replacing vertices with cliques.

\begin{theorem}\label{chi-bounded meta-theorem}
  Let $\mathcal G$ be a class of graphs closed under taking blowups such that $\chi(G) \leq c\cdot \omega(G)$ for all $G \in \mathcal G$, for some $c \in \mathbb R$.
  If $G \in \mathcal G$ and $f$ is a demand function for $G$ such that $f(v) \leq 1/(c\cdot\omega(v))$ for each $v \in V(G)$, then $G$ has an $f$-coloring.
  In particular, if $\mathcal G$ is a $\chi$-bounded class of graphs with $\chi$-binding function $g(n) = c\cdot n$ and $\mathcal G$ is closed under taking blowups, then $\mathcal G$ is local-fractionally $\chi$-bounded by $g$.  
\end{theorem}

Theorem~\ref{chi-bounded meta-theorem} has a number of corollaries.  In particular, note that Theorem~\ref{local frac perfect} follows from Theorem~\ref{chi-bounded meta-theorem}.  
In~\cite[Theorem 2.6]{CRRS16}, Brause et al.\ proved that the weaker version of Conjecture~\ref{local reeds conj} for the independence number holds for perfect graphs.  Theorem~\ref{local frac perfect} generalizes this result to its weighted version and also improves the bound.  It also confirms Conjecture~\ref{local reeds conj} for perfect graphs.  Theorem~\ref{local frac perfect} also implies that if $G$ is a bipartite graph, then the line graph of $G$ has an $f$-coloring if $f(uv) \leq 1/\max\{d(u), d(v)\}$ for each $uv\in E(G)$, which could be considered the local demands version of the fractional relaxation of K\" onig's Line Coloring Theorem.

A graph is \textit{quasiline} if the neighborhood of every vertex is the union of two cliques.  Since quasiline graphs are closed under taking blowups, by combining Theorem~\ref{chi-bounded meta-theorem} with the main result of~\cite{CO07}, we obtain the following.
\begin{corollary}\label{quasiline chi-bounded}
  If $G$ is a quasiline graph with demand function $f$ such that for each $v\in V(G)$, $f(v) \leq 2/(3\omega(v))$, then $G$ has an $f$-coloring.
\end{corollary}

The \textit{claw} is the graph $K_{1, 3}$.  Since claw-free graphs with independence number at least three are closed under taking blowups, by combining Theorem~\ref{chi-bounded meta-theorem} with the main result of~\cite{CS10}, we obtain the following.
\begin{corollary}\label{claw-free chi-bounded}
  If $G$ is a claw-free graph with independence number at least three, and if $f$ is a demand function for $G$ such that for each $v\in V(G)$, $f(v) \leq 1 / (2\omega(v))$, then $G$ has an $f$-coloring.
\end{corollary}

It is natural to ask if there are other natural local-fractionally $\chi$-bounded classes of graphs, and if so, what is their optimal local-fractional $\chi$-binding function.  There are several classes mentioned at the beginning of this section that would be interesting to explore.  In Section~\ref{edge coloring section}, we prove such results about line graphs, and Conjecture~\ref{local fractional TCC} concerns total graphs.  

\subsection{Proof of Theorem~\ref{chi-bounded meta-theorem}}

The remainder of this section is devoted to the proof of Theorem~\ref{chi-bounded meta-theorem}.  First we need the following lemma.
\begin{lemma}\label{clique size of blowup lemma}
  Let $G$ be a graph with demand function $f$, and let $N, c > 0$.  If for each $v\in V(G)$, we have $f(v) = 1/(c\omega(v))$, and if $N$ is a common denominator for $f$, then the graph obtained from $G$ by replacing each vertex with a clique of size $N\cdot f(v)$ contains no clique of size greater than $N / c$.
\end{lemma}
\begin{proof}
  Let $G'$ be the graph obtained from $G$ by replacing each vertex  with a clique of size $N\cdot f(v)$, and for each $v\in V(G)$, let $K_v$ denote the clique in $G'$ replacing $v$.  Let $K\subseteq V(G')$ be a clique in $G'$, and let $X = \{v\in V(G) : K_v\cap K \neq \varnothing\}$.  Since $K$ is a clique in $G'$, we have that $X$ is a clique in $G$.  Hence, for each $v\in X$, we have that $|K_v| \leq N / (c|X|)$.  Note that $|K| \leq \sum_{v\in X}|K_v|$.  Therefore $|K| \leq N / c$, so $G'$ contains no clique of size greater than $N / c$, as desired.
\end{proof}

Now the proof of Theorem~\ref{chi-bounded meta-theorem} follows easily.
\begin{proof}[Proof of Theorem~\ref{chi-bounded meta-theorem}]
  Let $G \in \mathcal G$, and let $f$ be a demand function for $G$ such that $f(v) = 1 / (c \cdot \omega(v))$ for each $v\in V(H)$.  Let $N$ be a common denominator for $f$, let $G'$ be the blowup of $G$ obtained by replacing each vertex of $G$ with a clique of size $N\cdot f(v)$, and note that $G$ has an $(f, N)$-coloring if and only if $G'$ has chromatic number at most $N$.  Hence, to show that $G$ has an $f$-coloring, by Theorem~\ref{equivalent definitions}, it suffices to show that $G'$ has chromatic number at most $N$.  By Lemma~\ref{clique size of blowup lemma}, $\omega(G') \leq N / c$.  Since $\mathcal G$ is closed under taking blowups, we have $G' \in \mathcal G$, and it follows that $\chi(G') \leq c \cdot \omega(G') \leq N$, as desired.

  Now suppose $\mathcal G$ is a $\chi$-bounded class of graphs with $\chi$-binding function $g(n) = c\cdot n$ that is closed under taking blowups, let $G \in \mathcal G$, and let $H$ be an induced subgraph of $G$.  Every blowup $H'$ of $H$ is an induced subraph of a blowup of $G$ and thus satisfies  $\chi(H') \leq g(\omega(H'))$.  Hence, by applying the previous result with the class of induced subgraphs of a graph in $\mathcal G$ playing the role of $\mathcal G$, the graph $H$ has a $g$-coloring, so $\mathcal G$ is local-fractionally $\chi$-bounded by $g$, as desired.
\end{proof}

\section{Edge-Coloring}\label{edge coloring section}

In this section we consider fractional edge-coloring with local demands.  Vizing's Theorem~\cite{V64} states that every graph $G$ can be edge-colored with at most $\Delta(G) + 1$ colors.  Equivalently, it states that every graph $G$ satisfies $\chi(L(G)) \leq \Delta(G) + 1$, where $L(G)$ is the line graph of $G$.  
Vizing's Theorem can be generalized to multigraphs, as follows.  Every multigraph $G$ can be edge-colored with at most $\Delta(G) + \max_{uv\in E(G)}\edgemult{uv}$ colors, where $\Delta(G)$ is the maximum degree of the underlying simple graph and $\edgemult{uv}$ is the \textit{multiplicity} of the edge $uv$, or the number of edges in $G$ incident to both $u$ and $v$.  In this section, if $G$ is a multigraph and $v\in V(G)$, we use $|N(v)|$ to denote the number of neighbors of $v$ and let $d(v) = \sum_{u\in N(v)} \edgemult{uv}$.  In this section, we prove Theorem~\ref{fractional generalized vizings}, which is the local demands version of this generalization of Vizing's Theorem.

We also prove the following local demands version of a theorem of Shannon~\cite{S49}.

\begin{theorem}[Local Fractional Shannon's]\label{fractional shannons}
  If $G$ is a multigraph and $f$ is a demand function for $L(G)$ such that each $e\in V(L(G))$ with $e = uv\in E(G)$ satisfies $f(e) \leq 2 / (3\max\{d(u), d(v)\})$, then $L(G)$ has an $f$-coloring.
\end{theorem}

Theorem~\ref{fractional shannons} follows from Corollary~\ref{quasiline chi-bounded}; however, since Corollary~\ref{quasiline chi-bounded} relies on the results of~\cite{CO07}, we provide a more direct proof of Theorem~\ref{fractional shannons} in this section.

In order to prove each of these theorems, we will need Edmonds' Matching Polytope Theorem~\cite{E65}, which, in light of Theorem~\ref{equivalent definitions}, characterizes all of the fractional colorings of a line graph.
\begin{theorem}[Edmonds' Matching Polytope Theorem \cite{E65}]\label{edmonds matching polytope thm}
  If $G$ is a simple graph and $f$ a demand function for $L(G)$, then $L(G)$ has an $f$-coloring if and only if for all $v\in V(G)$,
  \begin{equation}
    \label{mp vertex constraint}
    \sum_{u\in N(v)}f(uv) \leq 1,
  \end{equation}
  and for every $S\subseteq V(G)$,
  \begin{equation}
    \label{mp subgraph constraint}
    \sum_{e\in E(G[S])} f(e) \leq \lfloor |S| / 2\rfloor.
  \end{equation}
\end{theorem}

In order to show that~\eqref{mp subgraph constraint} holds, we need the following lemma.
\begin{lemma}\label{subgraph constraint lemma}
  If $G$ is a simple graph, then
  \begin{equation*}
    \sum_{v\in V(G)}\frac{d(v)}{d(v) + 1} \leq |V(G)| - 1.
  \end{equation*}
\end{lemma}
\begin{proof}  
  We use induction on $|V(G)|$.  If $|V(G)| = 1$, then there are no edges and the lemma follows.  Therefore we may assume $|V(G)| > 1$.

  Let $v\in V(G)$ have minimum degree.  By induction,
  \begin{equation*}
    \sum_{u\in V(G - v)}\frac{d_{G - v}(u)}{d_{G - v}(u) + 1} \leq |V(G)| - 2.
  \end{equation*}
  Therefore it suffices to show that
  \begin{equation}\label{edge coloring inductive step}
    \frac{d_G(v)}{d_G(v) + 1} + \sum_{u\in N(v)} \frac{d_G(u)}{d_G(u) + 1} - \frac{d_{G - v}(u)}{d_{G-v}(u) + 1} \leq 1.
  \end{equation}
  If $u\in N(v)$, then $d_{G - v}(u) = d_G(u) - 1$.  Hence,
  \begin{equation*}
    \frac{d_G(u)}{d_G(u) + 1} - \frac{d_{G - v}(u)}{d_{G-v}(u) + 1} = \frac{1}{d_G(u)(d_G(u) + 1)} \leq \frac{1}{d_G(v)(d_G(v) + 1)}.
  \end{equation*}
  Therefore
  \begin{equation*}
    \sum_{u\in N(v)}\frac{d_G(u)}{d_G(u) + 1} - \frac{d_{G - v}(u)}{d_{G-v}(u) + 1}  \leq \frac{1}{d_G(v) + 1},
  \end{equation*}
  and \eqref{edge coloring inductive step} follows, as required.
\end{proof}

Before proving Theorem~\ref{fractional generalized vizings}, we show that the proof essentially reduces to the case of simple graphs using the following lemma.
\begin{lemma}\label{multigraph to simple lemma}
  If $G$ is a multigraph and $v\in V(G)$, then
  \begin{equation*}
    \sum_{u\in N(v)}\frac{\edgemult{uv}}{d(v) + \edgemult{uv}} \leq \frac{|N(v)|}{1 + |N(v)|}.
  \end{equation*}
\end{lemma}
\begin{proof}
  Note that by definition, $d(v) = \sum_{u\in N(v)}\edgemult{uv}$.  Note also that $\frac{x}{d(v) + x}$ is concave as a function of $x$, so by Jensen's Inequality,
  \begin{equation*}
    \frac{\sum_{u\in N(v)}\edgemult{uv}/(d(v) + \edgemult{uv})}{|N(v)|} \leq \frac{d(v)/|N(v)|}{d(v) + d(v)/|N(v)|} = \frac{1/|N(v)|}{1 + 1/|N(v)|}. 
  \end{equation*}
  Rearranging terms,
  \begin{equation*}
    \sum_{u\in N(v)}\edgemult{uv}/(d(v) + \edgemult{uv}) \leq \frac{|N(v)|}{1 + |N(v)|},
  \end{equation*}
  as desired.
\end{proof}

We can now prove Theorem~\ref{fractional generalized vizings}.
\begin{proof}[Proof of Theorem~\ref{fractional generalized vizings}]
  Let $G'$ be the underlying simple graph of $G$, and let $f'$ be a demand function for $L(G')$ such that for each $e\in V(L(G'))$ where $e = uv\in E(G')$, we have $f'(e) = \edgemult{uv}/(\max\{d_G(u), d_G(v)\} + \edgemult{uv})$.  It suffices to show that $L(G')$ has an $f'$-coloring, because then $L(G)$ has an $f$-coloring, as desired.  
  By Theorem~\ref{edmonds matching polytope thm}, it suffices to show that~\eqref{mp vertex constraint} and \eqref{mp subgraph constraint} hold for $G'$ and $f'$.

  For each $v\in V(G)$ and $u\in N(v)$, we have $f'(uv) \leq \edgemult{uv}/(d_G(v) + 1)$.  Hence, for each $v\in V(G)$, we have $\sum_{u\in N(v)}f'(uv) \leq d_G(v)/(d_G(v) + 1) \leq 1$, so \eqref{mp vertex constraint} holds, as desired.

  Let $S\subseteq V(G)$, and note that
  \begin{equation*}
    2\sum_{e\in E(G'[S])}f'(e) = \sum_{v\in S}\sum_{u\in N(v)\cap S}f'(uv) \leq \sum_{v\in S}\sum_{u \in N(v)\cap S}\frac{\edgemult{uv}}{d_{G[S]}(v) + 1}.
  \end{equation*}
  By Lemma~\ref{multigraph to simple lemma} applied to each $v\in S$,
  \begin{equation*}
    \sum_{v\in S}\sum_{u\in N(v)\cap S}\frac{\edgemult{uv}}{d_{G[S]}(v) + 1} \leq \sum_{v\in S}\frac{|N(v)\cap S|}{1 + |N(v)\cap S|} = \sum_{v\in S} \frac{d_{G'[S]}(v)}{1 + d_{G'[S]}(v)}.
  \end{equation*}
  By Lemma~\ref{subgraph constraint lemma}, the right side of the previous inequality is at most $|S| - 1$, so the previous two inequalities imply that
  \begin{equation*}
    \sum_{e\in E(G[S])}f(e) \leq \frac{|S| - 1}{2} \leq \left\lfloor \frac{|S|}{2}\right\rfloor,
  \end{equation*}
  as required.
\end{proof}

We conclude this section with the proof of Theorem~\ref{fractional shannons}.
\begin{proof}[Proof of Theorem~\ref{fractional shannons}]
  Let $G'$ be the underlying simple graph of $G$, and let $f'$ be a demand function for $L(G')$ such that for each $e\in V(L(G'))$ where $e = uv\in E(G')$, we have $f'(e) = 2\edgemult{uv}/(3\max\{d_G(u), d_G(v)\})$.  It suffices to show that $L(G')$ has an $f'$-coloring, because then $L(G)$ has an $f$-coloring, as desired.  
  By Theorem~\ref{edmonds matching polytope thm}, it suffices to show that~\eqref{mp vertex constraint} and \eqref{mp subgraph constraint} hold for $G'$ and $f'$.

  For each $v\in V(G)$ and $u\in N(v)$, we have $f'(uv) \leq 2\edgemult{uv}/(3d_G(v))$.  Hence, for each $v\in V(G)$, we have $\sum_{u\in N(v)}f'(uv) \leq 2d_G(v)/(3d_G(v)) = 2/3 \leq 1$, so \eqref{mp vertex constraint} holds, as desired.  Moreover, for any $S\subseteq V(G)$,
  \begin{equation*}
    2\sum_{e\in E(G'[S])}f'(e) \leq \sum_{v\in S}\sum_{u\in N(v)}f'(uv) \leq \sum_{v\in S}2/3 = 2|S|/3.
  \end{equation*}
  Therefore $\sum_{e\in E(G'[S])}f'(e) \leq |S|/3$, so if $|S| \geq 2$, then $\sum_{e\in E(G'[S])}f'(e) \leq \lfloor |S| / 2 \rfloor$.  Hence, \eqref{mp subgraph constraint} holds, as required.
\end{proof}

\section*{Acknowledgements}

We thank the referee for their careful reading of this paper and their suggestions.


\end{document}